\documentclass[reqno,11pt]{amsart}
\usepackage{setspace}
\makeatletter

\usepackage[T1]{fontenc}
\usepackage{lmodern}

\usepackage{fullpage}
\usepackage{textcmds}  
\usepackage{amsmath, amssymb, amsfonts, amstext, verbatim, amsthm, mathrsfs, stmaryrd}
\usepackage{thmtools}
\usepackage{appendix}
\usepackage{mathdots}
\usepackage{ytableau}
\usepackage[all,cmtip,2cell]{xy}
\usepackage{pgf,tikz,pgfplots}
\pgfplotsset{compat=1.15}
\usetikzlibrary{arrows}
\usetikzlibrary{positioning}
\usetikzlibrary{quotes}
\usepackage{tikz-cd}
\usepackage{quiver}
\usepackage{mathtools}
\usepackage[dvipsnames]{xcolor}

\usepackage{graphicx}
\graphicspath{ {images/} }

\tikzset{
  symbol/.style={
    draw=none,
    every to/.append style={
      edge node={node [sloped, allow upside down, auto=false]{$#1$}}}
  }
}

\usepackage{enumerate}
\usepackage{enumitem}
\setlist{topsep=1em, itemsep=1em}
\usepackage[colorlinks=true,linkcolor=blue,citecolor=blue,urlcolor=blue,citebordercolor={0 0 1},urlbordercolor={0 0 1},linkbordercolor={0 0 1}]{hyperref} 
\usepackage[shortalphabetic]{amsrefs}
\usepackage[nameinlink, capitalize]{cleveref}

\usepackage{enotez}
\setenotez{backref=true}

\def\mydefc#1{\expandafter\def\csname c#1\endcsname{\mathcal{#1}}}
\def\mydefallc#1{\ifx#1\mydefallc\else\mydefc#1\expandafter\mydefallc\fi}
\mydefallc ABCDEFGHIJKLMNOPQRSTUVWXYZ\mydefallc

\def\mydefb#1{\expandafter\def\csname b#1\endcsname{\mathbf{#1}}}
\def\mydefallb#1{\ifx#1\mydefallb\else\mydefb#1\expandafter\mydefallb\fi}
\mydefallb ABCDEFGHIJKLMNOPQRSTUVWXYZ\mydefallb

\def\mydeffrac#1{\expandafter\def\csname frac#1\endcsname{\mathscr{#1}}}
\def\mydeffracall#1{\ifx#1\mydeffracall\else\mydeffrac#1\expandafter\mydeffracall\fi}
\mydeffracall ABCDEFGHIJKLMNOPQRSTUVWXYZ\mydeffracall

\theoremstyle{plain}
\newtheorem{thm}{Theorem}[section]
\newtheorem{cor}[thm]{Corollary}
\newtheorem{lem}[thm]{Lemma}

\newtheorem{prop}[thm]{Proposition}
\newtheorem{quest}[thm]{Question}

\theoremstyle{definition}
\newtheorem{rem}[thm]{Remark}
\newtheorem{defn}[thm]{Definition}

\newtheorem{ex}[thm]{Example}

\newtheorem{introthm}{Theorem}

\crefname{thm}{theorem}{theorems}
\crefname{cor}{corollary}{corollaries}
\crefname{lem}{lemma}{lemmas}
\crefname{prop}{proposition}{propositions}
\crefname{defn}{definition}{definitions}
\crefname{rem}{remark}{remarks}
\crefname{claim}{claim}{claims}
\crefname{conj}{conjecture}{conjectures}
\crefname{ex}{example}{examples}
\crefname{fact}{fact}{facts}
\crefname{quest}{question}{questions}

\Crefname{thm}{Theorem}{Theorems}
\Crefname{cor}{Corollary}{Corollaries}
\Crefname{lem}{Lemma}{Lemmas}
\Crefname{prop}{Proposition}{Propositions}
\Crefname{defn}{Definition}{Definitions}
\Crefname{rem}{Remark}{Remarks}
\Crefname{claim}{Claim}{Claims}
\Crefname{conj}{Conjecture}{Conjectures}
\Crefname{ex}{Example}{Examples}
\Crefname{fact}{Fact}{Facts}
\Crefname{quest}{Question}{Questions}

\def\rm{\mathrm}
\def\d{\mathrm{d}}

\DeclareMathOperator{\supp}{supp}

\newcommand{\DCoh}{\mathcal{D}^{b}}

\DeclareMathOperator{\Cone}{Cone}

\DeclareMathOperator{\ev}{ev}

\DeclareMathOperator{\Hom}{Hom}

\newcommand{\pt}{\mathrm{pt}}

\DeclareMathOperator{\Stab}{Stab}
\DeclareMathOperator{\Sym}{Sym}

\DeclareMathOperator{\Pic}{Pic}

\DeclareMathOperator{\fib}{fib}
\DeclareMathOperator{\im}{im}
\DeclareMathOperator{\HN}{HN}

\DeclareMathOperator{\coker}{coker}

\DeclareMathOperator{\Slice}{Slice}

\DeclareMathOperator{\rep}{rep}

\DeclareMathOperator{\logZ}{logZ}
\DeclareMathOperator{\gr}{gr}

\def\bf{\mathbf}

\def\Astab{\mathcal{A}\Stab}

\usepackage{pbox}
\usepackage[normalem]{ulem}

\makeatletter

\usepackage{babel}
\begin{document}

\title[Properties of deformed mass and phase functions]{Properties of deformed mass and phase functions}

\author[D. Halpern-Leistner]{Daniel Halpern-Leistner}
\address{Department of Mathematics, Cornell University, Ithaca, NY}
\email{daniel.hl@cornell.edu}

\author[A. Robotis]{Antonios-Alexandros Robotis}
\address{Department of Mathematics, Columbia University, New York, NY}
\email{a.robotis@columbia.edu}

\begin{abstract}
    We establish basic properties of the deformed mass and phase functions on the space of stability conditions. We prove that these functions are continuous and deduce that the space of stability conditions admits a homeomorphic embedding into a product space of finite measures. Sub\-sequently, we give a proof of the triangle inequality for deformed mass functions and provide estimates for the deformed mass of truncations of objects with respect to a slicing.
\end{abstract}

\maketitle

\tableofcontents

\addtocontents{toc}{\protect\setcounter{tocdepth}{1}}

\linespread{1.15}\selectfont

\section{Introduction}

Since their introduction by Bridgeland \cite{Br07} over twenty years ago, stability conditions on triang\-ulated categories have received considerable attention in a number of contexts, including algebraic geometry, representation theory, and mathematical physics. Remarkably, the set of stability conditions on a triangulated category $\cD$, denoted $\Stab(\cD)$, has a canonical metric topology. Further\-more, subject to finiteness hypotheses it admits a local homeomorphism to a complex vector space, giving it a canonical complex manifold structure.

Whenever it is non-empty, the space of stability conditions is not compact. In the concomitant work \cite{augmented}, we introduce a partial compactification $\Astab(\cD)$ of the quotient space $\Stab(\cD)/\bf{C}$ whose boundary points parametrize new categorical structures on $\cD$, called \emph{augmented stability conditions}. The present work, which can be read independently of \cite{augmented}, establishes basic prop\-erties of certain functions defined on $\Stab(\cD)$, called \emph{deformed mass} and \emph{deformed phase} functions --- see \Cref{S:deformed_mass_and_phase} for definitions. These functions are used crucially in \cite{augmented} to define convergence in the space of augmented stability conditions. 

Deformed mass functions were introduced in work of  Dimitrov--Haiden--Katzarkov--Kontsevich \cite{DynamicsDHKK} on categorical dynamics and the properties of these functions were further developed by Ikeda \cite{ikeda_2021}. On the other hand, we introduce deformed phase as an extension of the average phase functions studied in \cite{quasiconvergence}. We next summarize the contents of the present work. In this introduction, we assume that all stability conditions satisfy the support property of \cite{KS08}, however this hypothesis can be weakened in certain places --- see \Cref{D:stabilitycondition} and \Cref{R:weakenedhypothesis}.

Given a (pre-)stability condition $\sigma$ on a triangulated category $\cD$, we assign to each non-zero object of $\cD$ a measure $\d m_{\sigma,E}$ on $\bf{R}$, called the associated \emph{mass measure} --- see \Cref{D:massmeasures}. This is a finitely supported measure which records the phases and masses of the Harder--Narasimhan factors of $E$ with respect to $\sigma$. It is natural to ask what functions on $\Stab(\cD)$ can be defined by integration against these measures: both the deformed mass and phase functions are of this type.

\begin{introthm}
[= \Cref{T:jointlycontinuous}] 
\label{T:introA}
     Given a non-empty subspace $\Omega \subset \bf{R}$ and $F\in \mathscr{C}^0(\Omega\times \bf{R})$, then for all non-zero objects $E$ of $\cD$, the function $\Omega \times \Stab(\cD)\to \bf{R}$ defined by 
     \[
        (s,\sigma)\mapsto \int_{\bf{R}} F(s,\theta)\:\d m_{\sigma,E}(\theta)
     \]
     is continuous.
\end{introthm}

An immediate consequence of this is that the deformed mass and phase functions associated to a given $E$ are continuous functions on $\bf{R}\times \Stab(\cD)$ --- see \Cref{C:continuityoffunctions}. Beyond these examples, \Cref{T:introA} gives a simple recipe for producing continuous functions on $\Stab(\cD)$. 

Let $\mathscr{M}(\bf{R})$ denote the space of finitely supported measures on $\bf{R}$, equipped with the weak topology dual to $\mathscr{C}^0(\bf{R})$ and let $\mathfrak{d}$ denote the set of isomorphism classes of non-zero objects of $\cD$.\footnote{Strictly speaking, we assume here that $\cD$ is essentially small so that $\mathfrak{d}$ is a set. This is true for all of the examples considered in practice --- see \cite{stacks-project}*{\href{https://stacks.math.columbia.edu/tag/077K}{Tag 077K}} for a related discussion.} Let $\d m$ denote the map $\Stab(\cD) \to \mathscr{M}(\bf{R})^\mathfrak{d}$ sending $\sigma \mapsto (\d m_{\sigma,E})_{E\in \mathfrak{d}}$.

\begin{introthm}
[= \Cref{C:continuousmaptomeasures} + \Cref{T:measuretheoreticembedding}]
\label{T:introB}
The map $\d m$ is continuous, as is the composite map $\ell$ in the diagram below:
\begin{equation*} 
    \begin{tikzcd}
        \Stab(\cD) \arrow[dr,dashed,"\ell",swap] \arrow[r,"\mathrm{d}m"] & \mathscr{M}(\bf{R})^\mathfrak{d}\arrow[d,"\log\int_{\bf{R}} (-) + i\pi \int_{\bf{R}}\theta(-)"]\\
        & \bf{C}^{\mathfrak{d}}.
    \end{tikzcd}
\end{equation*}
    Furthermore, both $\d m$ and $\ell$ are injective. Finally, $\ell$ induces a continuous injection $\overline{\ell}:\Stab(\cD)/\bf{C} \hookrightarrow \bf{C}^{\mathfrak{d}}/\bf{C}$, where $\bf{C}$ acts by componentwise translation on $\bf{C}^{\mathfrak{d}}$.
\end{introthm}


A defect of \Cref{T:measuretheoreticembedding} is that the maps it constructs from $\Stab(\cD)$ are not proveably homeo\-morphisms onto their images. This difficulty stems from the fact that the upper and lower phase functions $\phi^+(E)$ and $\phi^-(E)$ do not admit continuous extensions to $\mathscr{M}(\bf{R})$ with the weak topology --- see \Cref{L:upperlowernotcont}.

We define a different topology on $\mathscr{M}(\bf{R})$ by interpreting it as a limit of configuration spaces of weighted points in $\bf{R}$. We refer to this as the \emph{strong} topology on $\mathscr{M}(\bf{R})$. A key property of this topology is that the functions $\phi^{\pm}(E)$ admit continuous extensions to $\mathscr{M}(\bf{R})$ by \Cref{L:continuousfunctionsstrong}. Using this, we define a strong topology on $\mathscr{M}(\bf{R})^{\mathfrak{d}}$, which is a strengthening of the product topology coming from equipping $\mathscr{M}(\bf{R})$ with the strong topology.  

\begin{introthm}
[= \Cref{T:strongembedding}]
\label{T:introC}
    The map $\d m:\Stab(\cD) \to \mathscr{M}(\bf{R})^{\mathfrak{d}}$ is a homeomorphism onto its image when $\mathscr{M}(\bf{R})^{\mathfrak{d}}$ is equipped with the strong topology.
\end{introthm}

Regarding $\mathscr{M}(\bf{R})$ as a limit of configuration spaces of points also gives it a natural smooth structure, in the generalized sense proposed by Souriau \cite{SouriauGroupes}. We discuss this briefly in \Cref{S:continuityproperties}, noting that deformed mass functions are smooth with respect to this smooth structure, whereas this is not the case for the usual manifold structure on $\Stab(\cD)$. 

The latter part of the paper is dedicated to proving certain fundamental inequalities for the deformed mass function. Among these is a strengthened version of the triangle inequality for deformed mass functions, which is used to establish Theorems \ref{T:introA}, \ref{T:introB}, and \ref{T:introC}.

\begin{introthm}
[= \Cref{L:triangle_inequality} + \Cref{C:reverse_triangle_inequality}]
\label{T:introD}
For any exact triangle $E \to F \to G$ in $\cD$ and any real numbers $a,t$ one has 
\begin{equation}
    0\le m_\sigma^t(E) + m_\sigma^t(G) - m_\sigma^t(F) \le (1+e^{\lvert t\rvert})\left(m_\sigma^t(G^{>a}) + m_\sigma^t(E^{\le a})\right).
\end{equation}
\end{introthm}

Here, $G^{>a}$ denotes the part of $G$ lying in $\cP_\sigma(a,\infty)$, defined using the bounded t-structure $(\cP_\sigma(a,\infty), \cP_\sigma(-\infty,a])$ for any $a\in \bf{R}$, and $E^{\le a}$ is defined analogously. The triangle inequality 
\[
    m_\sigma^t(F) \le m_\sigma^t(E) + m_\sigma^t(G)
\]
was already stated as \cite{ikeda_2021}*{Prop. 3.3}, however the proof given there contains a gap. In \Cref{S:triangleinequality}, we explain the gap and complete the proof. 

The last section is dedicated to proving a technical result, which estimates the deformed masses of truncations of objects --- see \Cref{T:filtration_inequality}. This is used crucially in \cite{augmented}, but may be of independent interest and utility.

\section*{Acknowledgments}

We are grateful to Hannah Dell, David Ploog, and Francesco Sala for helpful conversations on topics related to the present work. We are especially thankful to Asilata Bapat, Anand Deopurkar, and Anthony Licata for communicating to us an argument which greatly simplified the proof of \Cref{P:continuity} and thereby \Cref{T:jointlycontinuous}. Special thanks are also due to Nicol\'{a}s Vilches for providing many useful comments and improvements on a draft of this article.

The authors appreciate the support of the NSF CAREER grant DMS-1945478, and NSF FRG grant DMS-2052936, and the support and hospitality of the Simons Laufer Mathematical Sciences Institute during the preparation of this paper. The first author was also supported by an Alfred P. Sloan Foundation Research Fellowship (FG-2022-18834), and a Simons Foundation Fellowship. The second author was also supported by NSF
grant DMS-250340. He also thanks Maria Teresa for her love and support. 

\section{Definitions}
\label{S:deformed_mass_and_phase}

Let $\cD$ denote a $k$-linear triangulated category, for $k$ a field. To fix conventions, we recall the definition of (pre-)stability conditions on triangulated categories \cites{Bayer_short,Br07}. 

\begin{defn}
\label{D:stabilitycondition}
    A \emph{pre-stability condition} on $\cD$ is a pair $(Z,\cP)$ where \vspace{-2mm}
    \begin{enumerate}  
        \item $Z \in \Hom_{\bf{Z}}(\rm{K}_0(\cD),\bf{C})$ is called the \emph{central charge}; and \vspace{-2mm}
        \item $\cP = \{\cP(\phi):\phi \in \bf{R}\}$ is a collection of full additive subcategories of $\cD$, called a \emph{slicing} such that:\vspace{-3mm}
        \begin{enumerate}
            \item if $\phi>\psi$ then $\Hom_{\cD}(\cP(\phi),\cP(\psi)) = 0$; \vspace{-3mm}
            \item $\cP(\phi)[1] = \cP(\phi+1)$ for all $\phi \in \bf{R}$; and \vspace{-3mm}
            \item for all non-zero objects $E$ of $\cD$, there is a sequence of exact triangles
            \[
                \begin{tikzcd}
                    0 = E_0 \arrow[r] & E_1\arrow[d]\arrow[r] & \cdots\arrow[r]& E_{n-1}\arrow[r]& E\arrow[d]\\
                    & G_1 \arrow[ul,"+1"]& \cdots&&G_n \arrow[ul,"+1"]
                \end{tikzcd}
            \]
            where $G_i \in \cP(\phi_i)$ for all $i=1,\ldots, n$ and $\phi_1>\cdots>\phi_n$, called a Harder--Narasimhan (HN) filtration.
        \end{enumerate}
    \end{enumerate}

    Non-zero objects of $\cP(\phi)$ are called \emph{semistable} of phase $\phi$. The objects $\{G_i\}_{i=1}^n$ are called the Harder--Narasimhan factors of $E$ with respect to $\sigma$. It can be deduced from the definition that the HN factors are unique up to isomorphism \cites{Br07,GKRpaper04}. 
    
    We sometimes write $\{\HN_i^\sigma(E)\}$ for the \emph{Harder--Narasimhan factors} of $E$. A pre-stability condition is called a \emph{stability condition} if in addition each of the categories $\cP(\phi)$ is a finite length Abelian category. In this case, the Jordan--H\"older (JH) theorem for Abelian categories  \cite{stacks-project}*{\href{https://stacks.math.columbia.edu/tag/0FCD}{Tag 0FCD}} implies the existence of JH filtrations of non-zero objects of $\cP(\phi)$ by simple objects. The simple objects of $\cP(\phi)$ are called \emph{stable} of phase $\phi$. 
\end{defn}

\begin{rem}
\label{R:weakenedhypothesis}
    The definition of stability condition in this paper differs slightly from other definitions in the literature. In \cite{Br07}, Bridgeland imposes \emph{local finiteness} on the pair $(Z,\cP)$. On the other hand, following Kontsevich--Soibelman \cite{KS08}, the now-standard definition in \cite{Bayer_short} requires that $(Z,\cP)$ satisfy the \emph{support property} with respect to a fixed map $v:\rm{K}_0(\cD) \twoheadrightarrow \Lambda$ of Abelian groups, where $\Lambda$ is of positive and finite rank. Usually, $\Lambda$ is taken to be torsion free. A pair $(Z,\cP)$ as in \Cref{D:stabilitycondition} satisfies the support property if $Z$ factors through $v$ and for any norm $\lVert\:\cdot\:\rVert$ on $\Lambda_{\bf{R}}$, one has 
    \[
        \inf_{0\ne E \in \cP(\phi),\phi \in \bf{R}} \frac{\lvert Z(E)\rvert}{\lVert v(E)\rVert} > 0. 
    \]
    Both local finiteness and the support property imply that the categories $\cP(\phi)$ are all finite-length. In \Cref{S:continuityproperties}, stability conditions will be as in \Cref{D:stabilitycondition}, while in \Cref{S:strongerembedding} we will assume that stability conditions satisfy the support property. The results of the Sections \ref{S:triangleinequality} and \ref{S:estimatesfortruncations} require only pre-stability conditions. 
\end{rem}

If $E$ is a $\sigma$-semistable object, then its \emph{mass} is defined as $m_\sigma(E) = \lvert Z(E)\rvert$. If $E$ is an arbitrary non-zero object of $\cD$ with HN factors $G_1,\ldots, G_n$ as above, then its mass is $m_\sigma(E) = \sum_{i=1}^n m_\sigma(G_i)$ Remarkably, the space of (pre-)stability conditions has a metric topology, with generalized metric given by 
\begin{equation}
\label{E:Brmetric}
    d(\sigma,\tau) = \sup_{0\ne E \in \cD}\left\{ \lvert \phi_{\sigma_1}^+(E) - \phi_{\sigma_2}^+(E)\rvert, \lvert \phi_{\sigma_1}^-(E) - \phi_{\sigma_2}^-(E)\rvert, \left\lvert \log  \tfrac{m_{\sigma_2}(E)}{m_{\sigma_1}(E)}\right\rvert \right\}.
\end{equation}
See \cite{Br07}*{Prop. 8.1} for a proof and related discussions.


\begin{defn}
\label{D:massmeasures}
    Let $E$ be a non-zero object of $\cD$. The \emph{mass measure} associated to $E$ by $\sigma$ is the finitely supported measure 
    \begin{equation}
    \label{E:massmeasure}
        \rm{d}m_{\sigma,E} = \sum_{i=1}^n \lvert Z_\sigma(G_i)\rvert \cdot \delta_{\phi_\sigma(G_i)}
    \end{equation}
    on $\bf{R}$, where $\delta_x$ is the indicator measure of $\{x\}$.\endnote{That is, given a subset $A\subset \bf{R}$ we have 
\[
    \delta_{x}(A) = 
    \begin{cases}
        1 & x \in A \\
        0 & x \not\in A.
    \end{cases}
\]} Let $t \in \bf{R}$ be given; the $t$-\emph{deformed mass} of $E$ with respect to $\sigma$ is 
    \[
        m_\sigma^t(E) := \int_{\bf{R}} e^{t\theta}\d m_{\sigma,E}(\theta) = \sum_{i=1}^n  \lvert Z_\sigma(G_i)\rvert \cdot e^{t\theta_i}. 
    \]
    When $t=0$, $m_\sigma^0(E) = m_\sigma(E)$ is simply the mass of $E$ as defined above.
\end{defn}

The deformed mass was introduced by Dimitrov--Haiden--Katzarkov--Kontsevich \cite{DynamicsDHKK} in the context of categorical dynamics, under the name \emph{polynomial mass functions}. 

\begin{defn}
    Given $t\in \bf{R}$, the $t$-\emph{deformed (average) phase} of $E$ is 
    \[
        \phi_\sigma^t(E) = 
        \begin{cases}
            \frac{1}{m(E)}\sum_{i=1}^n \lvert Z(G_i)\rvert \cdot \theta_i & t=0\\
            \frac{1}{t}\log\left(\frac{m^t(E)}{m(E)}\right)& t\ne 0.
        \end{cases}
    \]
    When $t = 0$, we recover the average phase of $E$ as in \cite{quasiconvergence}*{\S2}. We write $\phi_\sigma^0(E) = \phi_\sigma(E)$. 
\end{defn}

The average phase is the expectation value of the random variable $\theta$ with respect to the probability measure $\d m_{\sigma,E}/m_\sigma(E)$. We will omit $\sigma$ from the notation when it is clear from the context. 

\begin{prop}
\label{P:firstproperties}
    The functions $\phi_\sigma^t(E)$ and $m_\sigma^t(E)$ satisfy the following properties:
    \vspace{-2mm}
    \begin{enumerate}
        \item $m_\sigma^t(E)$ is an analytic function of $t$; \vspace{-2mm}
        \item $m_\sigma(E)\cdot e^{t\phi_\sigma(E)}\le m^t_\sigma(E)$ and if $t\ne 0$ there is equality if and only if $E$ is $\sigma$-semistable;  \vspace{-2mm}
        \item $\phi_\sigma^t(E)$ is a monotonically increasing analytic function of $t \in \bf{R}$, which is strictly increasing unless $E$ is $\sigma$-semistable, in which case it is constant; and
        \vspace{-2mm}
        \item $\lim_{t\to\pm\infty} \phi_\sigma^t(E) = \phi^\pm_\sigma(E)$.\vspace{-2mm}
    \end{enumerate}
\end{prop}

\begin{proof}
    Property (1) is immediate, since $m^t(E)$ is a sum of functions of the form $e^{t\theta}$. For (2), since $e^{t \theta}$ is convex, Jensen's inequality applied to the probability measure $\d m_{E}/m(E)$ implies that 
    \begin{equation}\label{E:jensen}
        m(E) e^{t \phi(E)} \leq m^t(E).
    \end{equation}
    If $t \neq 0$, then $e^{t\theta}$ is strictly convex, so \eqref{E:jensen} is an equality if and only if $E$ is semistable. For (3), we first verify analyticity of $\phi_\sigma^t(E)$. For $t\ne 0$, $\phi_\sigma^t(E)$ is analytic being a composite of analytic functions. Next, write the $\sigma$-HN factors of $E$ as $G_1,\ldots, G_m$. Then, for $0<\lvert t\vert \ll 1$ we have
    \begin{align*}
        \phi_\sigma^t(E) & = \frac{1}{t} \log \left(1+ \frac{1}{m(E)} \left(\sum_{i=1}^m t\cdot \theta_i \cdot m(G_i) + \sum_{i=1}^m \sum_{n = 2}^\infty \frac{(t\cdot \theta_i)^n}{n!}m(G_i)\right)\right) \\
        & =\frac{1}{t}\left(\frac{1}{m(E)}\sum_{i=1}^m t\cdot \theta_i \cdot m(G_i) +O(t^2)\right)\\
        & = \phi_\sigma^0(E) + O(t).
    \end{align*}
    Since $\phi^t$ admits a natural extension to a holomorphic function on $\bf{C}^*$, this implies analyticity of $\phi_\sigma^t(E)$ at $0$. Next, applying Jensen's inequality again, this time to the probability measure $e^{t \theta} \d m_{E} / m^t(E)$, gives
    \[
        \frac{m(E)}{m^t(E)} = \int e^{-t \theta} \frac{e^{t\theta} \d m_E}{m^t(E)} \geq \exp\left(-t \int \theta \frac{e^{t\theta} \d m_E}{m^t(E)}\right) = \exp\left(\frac{-t}{m^t(E)} \frac{\partial m^t(E)}{\partial t}\right).
    \]
    Therefore, for $t\ne 0$
    \begin{equation} \label{E:jensen2}
        \frac{\partial \phi^t(E)}{\partial t} = \frac{1}{t^2} \left( \frac{t}{m^t(E)} \frac{\partial m^t(E)}{\partial t} - \log \left(\frac{m^t(E)}{m(E)}\right) \right) \geq 0,
    \end{equation}
with strict inequality unless $E$ is semistable. Thus, $\phi^t(E)$ is strictly increasing in $t$ if $E$ is not semistable, and constant in $t$ if $E$ is semistable. Finally, for (4) we compute 
\begin{align*}
    \lim_{t\to\infty} \phi^t(E)  & = \lim_{t\to \infty}\frac{\log(m^t(E)/m(E))}{t}\\
    & =  \lim_{t\to\infty} \frac{1}{m^t(E)}\frac{\partial m^t(E)}{\partial t} = \lim_{t\to\infty} \frac{\sum \theta_i \cdot e^{t\theta_i}m(F_i)}{\sum e^{t\theta_i}\cdot m(F_i)} = \phi^+(E). 
\end{align*}
The identity $\lim_{t\to-\infty} \phi^t(E) = \phi^-(E)$ is proven analogously.
\end{proof}

Finally, for any $t\in \bf{R}$, the \emph{log central charge functions} on $\cD$ are 
\begin{equation}
\label{E:ellt}
    \ell_\sigma^t(E) := \log(m_\sigma(E))+ i\pi \phi_\sigma^t(E).
\end{equation}
Note that only the imaginary part of $\ell_\sigma^t(E)$ depends on $t$. By \cite{Br07}*{Lem. 8.2}, the additive group $\bf{C}$ acts freely on $\Stab(\cD)$. The action is given by $w\cdot (Z,\cP) = (e^w\cdot Z,\cP^w)$, where $\cP^w(\phi) := \cP(\phi - \tfrac{\Im(w)}{\pi})$. It is an exercise to verify that 
\begin{equation}
\label{E:translation}
    \ell_{z\cdot \sigma}^t(E) = \ell_\sigma^t(E) + z
\end{equation}
for any $z\in \bf{C}$. The log central charge functions play a crucial role in the construction of the partial compactification of $\Stab(\cD)/\bf{C}$ in \cite{augmented}.

\section{Stability conditions as measure valued functions}
\label{S:continuityproperties}

Let $\mathscr{M}(\bf{R})$ denote the space of non-zero finitely supported non-negative measures on $\bf{R}$. An element $\mu$ of  $\mathscr{M}(\bf{R})$ can be written uniquely as 
\[
    \mu = \sum_{i=1}^m a_i \cdot \delta_{\theta_i}
\]
with $a_i\in \bf{R}_{>0}$ for all $i=1,\ldots, m$ and $\theta_1>\cdots>\theta_m$. In this notation, the support of $\mu$ is $\supp(\mu) = \{\theta_1,\ldots, \theta_m\}$. Next, let $\mathfrak{d}$ denote the set of isomorphism classes of non-zero objects of $\cD$. A pre-stability condition $\sigma$ on $\cD$ determines an element of $\mathscr{M}(\bf{R})^\mathfrak{d}$ by sending $E\in \mathfrak{d}$ to $\d m_{\sigma,E}$ as defined in \Cref{D:massmeasures}. Denote by 
\begin{equation}
    \rm{d}m:\Stab(\cD) \to \mathscr{M}(\bf{R})^\mathfrak{d}
\end{equation}
the corresponding map. In this section and the next, we will study properties of $\d m$. First, we establish some basic continuity properties. We will use the following elementary lemma several times:

\begin{lem}
\label{L:conevector}
    Let $0\le \epsilon<1$ be given. If $v_1,\ldots, v_n \in \bf{C}^*$ are contained in a cone emanating from the origin of angular width $\pi\epsilon$ radians then $\lvert \sum v_i\rvert/\sum \lvert v_i\rvert \in [1-\cos(\tfrac{\pi \epsilon}{2}),1]$.
\end{lem}

\begin{proof}
    Without loss of generality, we may assume that the cone is centered at $\bf{R}_{>0}\cdot i$. The result follows immediately from the inequalities $\cos(\tfrac{\pi\epsilon}{2})\sum \lvert v_i\rvert \le \left\lvert \sum v_i \right\rvert \le \sum \lvert v_i\rvert.$
\end{proof}

\begin{prop}
\label{P:continuity}
    For any $f\in \mathscr{C}^0(\bf{R})$ and $E\in \mathfrak{d}$ the map $\Stab(\cD)\to \bf{R}$ defined by $\sigma \mapsto \int_{\bf{R}} f\:\d m_{\sigma,E}$ is continuous.
\end{prop}

The following argument was explained to us by the authors of \cite{ThurstonBDL} and greatly simplifies our previous argument.

\begin{proof}
    Consider $\sigma \in \Stab(\cD)$ and suppose that $E$ is $\sigma$-semistable of some phase $\phi$. Let $\epsilon>0$ and $\tau \in \Stab(\cD)$ be given such that $d(\sigma,\tau)<\epsilon$. It follows that $E\in \cP_\tau(\phi-\epsilon,\phi+\epsilon)$. Let $m_+(\epsilon) := \max\{f(\theta):\theta \in [\phi-\epsilon,\phi+\epsilon]\}$ and define $m_-(\epsilon)$ analogously. Then, letting $\theta_1>\cdots>\theta_m$ denote the phases of the $\tau$-HN factors of $E$, we have
    \begin{align*}
        \int_{\bf{R}} f\:\d m_{\tau,E} = \int_{\phi-\epsilon}^{\phi+\epsilon}f\:\d m_{\tau,E} & = \sum_{i=1}^m f(\theta_i) \cdot \lvert Z_\tau(\HN_i^\tau(E))\rvert \\
        &\le m_+(\epsilon) \cdot \sum_{i=1}^m \lvert Z_\tau(\HN_i^\tau(E))\rvert \\
        & \le m_+(\epsilon) \cdot C(\epsilon) \cdot \lvert Z_\tau(E)\rvert
    \end{align*}
    where $C(\epsilon) \to 1$ as $\epsilon \to 0$. We obtain the constant $C(\epsilon)$ by applying \Cref{L:conevector} to $\{Z_\tau(\HN_i^\tau(E))\}_{i=1}^m$. Using similar reasoning, we obtain an analogous lower bound:
    \[
        m_-(\epsilon)\cdot c(\epsilon)\cdot \lvert Z_\tau(E)\rvert \le \int_{\bf{R}} f\:\d m_{\tau,E} \le m_+(\epsilon)\cdot C(\epsilon)\cdot \lvert Z_\tau(E)\rvert.
    \]
    As $\epsilon \to 0$, we have $m_{\pm}(\epsilon) \to f(\phi)$ and continuity at $\sigma$ follows, since by definition $\int_{\bf{R}} f\:\d m_{\sigma,E} = f(\theta) \cdot m_\sigma(E) = f(\theta)\cdot \lvert Z_\sigma(E)\rvert$ for $E$ semistable.

    Next, let $E$ be a non-semistable object of $\cD$ and let $\epsilon := \tfrac{1}{2}\min_i\left\{\lvert \HN_i^\sigma(E) - \HN_{i+1}^\sigma(E)\rvert\right\}$. Choose $\tau \in \Stab(\cD)$ such that $d(\sigma,\tau)<\epsilon$. It follows that 
    \[
        \d m_{\tau,E} = \sum_{i=1}^m \d m_{\tau,\HN_i^\sigma(E)}
    \]
    and the result now follows from the semistable case.
\end{proof}

\begin{thm}
\label{T:jointlycontinuous}
    Consider a non-empty subspace $\Omega \subset \bf{R}$ and $F\in \mathscr{C}^0(\Omega\times \bR)$. Then, for any $E\in \mathfrak{d}$, the map $\Omega \times \Stab(\cD)\to \bR$ given by
    \[
    (t,\sigma)\mapsto \int_{\bf{R}} F(t,\theta)\:\d m_{\sigma,E}(\theta)
    \]
    is continuous.
\end{thm}

\begin{proof}
    Fix $(s,\sigma)\in \Omega \times \Stab(\cD)$ and consider $(t,\tau)$ with $d := \max\{\lvert s-t\rvert, d(\sigma,\tau)\}$. We have 
    \begin{align*}
     \left\lvert \int F(s,\theta)\:\d m_{\sigma,E} - \int F(t,\theta)\:\d m_{\tau,E}\right \rvert & \le \left \lvert \int F(s,\theta)(\d m_{\sigma,E} - \d m_{\tau,E})\right \rvert  + \left\lvert \int F(s,\theta) - F(t,\theta) \: \d m_{\tau,E} \right \rvert
    \end{align*}
    where the first term on the right hand side tends to zero as $d(\sigma,\tau)\to 0$ by \Cref{P:continuity}. For the other term, note that there exists a compact set $K\subset \bf{R}$ such that $\supp(\d m_{\tau,E})\subset K$ for all $\tau$ with $d(\sigma,\tau)\le d$. Let $M_d = \max\{\lvert F(s,\theta) - F(t,\theta)\rvert: t\in [s-d,s+d], \theta \in K\}$. Then 
    \[
    \left\lvert \int F(s,\theta) - F(t,\theta)\: \d m_{\tau,E} \right\rvert \le \int \lvert F(s,\theta) - F(t,\theta)\rvert \:\d m_{\tau,E} \le M_d\cdot m_\sigma(E)\cdot e^d
    \]
    where $M_d\to 0$ as $d\to 0$.
\end{proof}

\begin{cor}
\label{C:continuityoffunctions}
    For any $E\in \mathfrak{d}$, the functions \vspace{-2mm}
    \begin{enumerate}
        \item $(t,\sigma)\mapsto m_\sigma^t(E)$; and \vspace{-2mm}
        \item $(t,\sigma)\mapsto \phi_\sigma^t(E)$\vspace{-2mm}
    \end{enumerate}
    define continuous maps $\bR \times \Stab(\cD)\to \bR$.
\end{cor}

\begin{proof}
    This follows from \Cref{T:jointlycontinuous}, since $m_\sigma^t(E) = \int e^{t\theta}\:\d m_{\sigma,E}(\theta)$, $\phi^0_\sigma(E) = \int \theta \:\d m_{\sigma,E}(\theta)$, and $\phi^t_\sigma(E) = \tfrac{1}{t}\log(m^t_\sigma(E)/m_\sigma(E))$ for $t\ne 0$.
\end{proof}

We write $\Stab_\Lambda(\cD)$ for the subspace of stability conditions satisfying the support property with respect to some map $v:\rm{K}_0(\cD)\to \Lambda$ as in \Cref{R:weakenedhypothesis}. \Cref{C:continuityoffunctions} has the following pleasant corollary: 

\begin{cor}
    The quotient map $\pi:\Stab_\Lambda(\cD)\to \Stab_\Lambda(\cD)/\bf{C}$ is a topologically trivial holo\-morphic principal $\bf{C}$-bundle. 
\end{cor}

\begin{proof}
    The action of $\bf{C}$ on $\Stab_\Lambda(\cD)$ is holomorphic and free. To see that $\pi$ is a holomorphic principal $\bf{C}$-bundle, let $x\in \Stab_\Lambda(\cD)/\bf{C}$ be given and consider $E\in \mathfrak{d}$ which is stable with respect to some, hence any, $\sigma \in x$. Choose a lift of $x$ to $\sigma \in \Stab_\Lambda(\cD)$. The support property implies that there is an open neighborhood $V$ of $\sigma$ on which $E$ is stable. Since $\pi$ is a quotient map, $\pi(V) =: U$ is open, and thus there exists an open set containing $x$ on which $E$ is stable. Next, define a section $U\to \Stab_\Lambda(\cD)$ by sending $y\in U$ to its unique representative $\sigma(y)$ such that $\logZ_{\sigma(y)}(E) = 0 + i\pi$. A coordinate calculation verifies that this map is holomorphic, so that $\pi$ is a holomorphic principal $\bf{C}$-bundle.
    
    For topological triviality, it suffices to construct a continuous section of $\pi$. Choose $F\in \mathfrak{d}$ and define a section $s_F$ of $\pi$ by sending $x\in \Stab_\Lambda(\cD)/\bf{C}$ to its unique representative $\sigma$ such that $\ell_\sigma(F) = 0 + i\pi$, where $\phi_\sigma(F)$ denotes the \emph{average} phase. To prove continuity, let $x$, $U$, and $E$ be as before, so that we have a trivialization $\pi^{-1}(U) \cong U\times \bf{C}$ as above. In this trivialization, $s_F(y) = (y,\ell_y(F) - \logZ_y(E))$, where we note that $\ell_y(F) - \logZ_y(E)$ is well-defined by \eqref{E:translation} and continuous by \Cref{C:continuityoffunctions}.
\end{proof}

Next, equip $\mathscr{M}(\bf{R})$ with the weak topology defined by $\mathscr{C}^0(\bf{R})$. That is, $(\mu_k)_{k\ge 1}\to \mu$ if and only if $\int f\:\d\mu_k\to \int f\:\d\mu$ for all $f\in \mathscr{C}^0(\bf{R})$. We consider $\mathscr{M}(\bf{R})^{\mathfrak{d}}$ with the induced product topology:

\begin{cor}
\label{C:continuousmaptomeasures}
    The map $\rm{d}m:\Stab(\cD) \to \mathscr{M}(\bf{R})^\mathfrak{d}$ is continuous.
\end{cor}

\begin{proof}
    The claim is equivalent to showing that for any given $E\in \mathfrak{d}$, any $f\in \mathscr{C}^0(\bf{R})$, and any convergent sequence $(\sigma_k)\to \sigma$ in $\Stab(\cD)$, we have $\lim_{k\to\infty} \int f\:\d m_{\sigma_k,E} = \int f\:\d m_{\sigma,E}$. This is precisely \Cref{P:continuity}.
\end{proof}

Next, we have a commutative diagram of maps, which are continuous by \Cref{P:continuity} and \Cref{C:continuousmaptomeasures}:
\begin{equation} 
\label{E:elldefinition}
    \begin{tikzcd}
        \Stab(\cD) \arrow[dr,dashed,"\ell^0",swap] \arrow[r,"\mathrm{d}m"] & \mathscr{M}(\bf{R})^\mathfrak{d}\arrow[d,"\log\int_{\bf{R}} (-) + i\pi \int_{\bf{R}}\theta(-)"]\\
        & \bf{C}^{\mathfrak{d}}.
    \end{tikzcd}
\end{equation}
For $t\ne 0$, we obtain a similar diagram for $\ell^t$ after setting the vertical map to be $\mu \mapsto \log \int_{\bf{R}}\d\mu + \tfrac{i\pi}{t} \log (\int_{\bf{R}}e^{t\theta}\:\d \mu /\int_{\bf{R}} \d\mu)$. See \eqref{E:ellt} for the definition of $\ell^t$. Note that the additive group $\bf{C}$ acts on $\bf{C}^{\mathfrak{d}}$ by $w\cdot (z_E)_{E \in \mathfrak{d}} = (w+z_E)_{E \in \mathfrak{d}}$.

\begin{thm}
\label{T:measuretheoreticembedding}
    For each $t\in \bf{R}$, the map $\ell^t:\Stab(\cD) \to \bf{C}^{\mathfrak{d}}$ is a continuous injection and induces a continuous injection $\overline{\ell}^t:\Stab(\cD)/\bf{C} \to \bf{C}^{\mathfrak{d}}/\bf{C}$ such that the following diagram commutes:
    \[
        \begin{tikzcd}
            \Stab(\cD) \arrow[r,"\ell^t",hook]\arrow[d]& \bf{C}^{\mathfrak{d}}\arrow[d,"\pi"]\\
            \Stab(\cD)/\bf{C} \arrow[r,"\overline{\ell}^t",hook]& \bf{C}^{\mathfrak{d}}/\bf{C}.
        \end{tikzcd}
    \]
    Consequently, $\d m:\Stab(\cD) \to \mathscr{M}(\bf{R})^{\mathfrak{d}}$ is injective.
\end{thm}

\begin{proof}
    We already know that $\ell^t$ is continuous by \Cref{C:continuousmaptomeasures}. We verify that $\ell^t$ is injective. We first show that an object $F$ of $\cD$ is stable if and only if the only exact triangles $E \to F \to G$ with $m_\sigma(F)=m_\sigma(E)+m_\sigma(G)$ have either $E \cong 0$ or $G \cong 0$. First, assume $F$ has this property. The mass function is additive on the Jordan--H\"older filtration of $F$, so $F$ must have trivial JH filtration, and is thus stable. 

    Conversely, assume $F$ is stable and that there is an exact triangle $E \to F \to G$ on which $m_\sigma$ is additive. Up to shifting phases, we may assume that $\phi(F)=1$. The long exact homology sequence for the triangle, with respect to the heart $\cP_\sigma(0,1]$ gives an exact sequence 
\[
    0 \to H_1(G) \to H_0(E) \to F \to H_0(G) \to H_{-1}(E) \to 0,
\]
along with isomorphisms $H_i(E) \cong H_{i+1}(G)$ for $i \neq 0,-1$. Because $F$ is simple of phase 1, $F$ has no non-trivial subobjects in $\cP_\sigma(0,1]$, so the map $H_0(E) \to F$ is either $0$ or surjective. We therefore have either $0 \to H_1(G) \to H_0(E) \to F \to 0$ or $0 \to F \to H_0(G) \to H_{-1}(E) \to 0$. We will explain the remainder of the proof in the case where the first short exact sequence holds and leave the second case for the reader, because the argument there is nearly identical.

The hypothesis that $m_\sigma(F) = m_\sigma(E)+m_\sigma(G)$ and the triangle inequality (\Cref{L:triangle_inequality}) imply that $m_\sigma(F) = m_\sigma(H_0(E))+m_\sigma(H_1(G))$ and that all other homology groups of $E$ and $G$ vanish.\endnote{The triangle inequality for masses implies $m_\sigma(F) \leq m_\sigma(H_0(E))+m_\sigma(H_1(G))$, and the right hand side is $\leq m_\sigma(E)+m_\sigma(G)$, so the equality $m_\sigma(F)=m_\sigma(E)+m_\sigma(G)$ implies that both inequalities are also equalities.} One can concatenate the HN curves\footnote{Given a non-zero object $A$ of the heart $\cP_\sigma(0,1]$ with HN factors $\HN_\sigma^1(A),\ldots, \HN_\sigma^k(A)$, the associated HN curve is the piecewise linear path beginning at the origin defined by the points $0,Z(\HN_\sigma^1(A)),\ldots, Z(\HN_\sigma^k(A))$. See the black curve in \Cref{fig:truncated_triangle} for a visualization.} of $H_1(G)$ and $H_0(E)$ to obtain a piecewise linear path connecting $Z(H_1(G))$ and $Z(H_0(E))$ in $\bC$, and whose length in $m_\sigma(H_0(E))+m_\sigma(H_1(G))$. On the other hand, the HN curve of $F$ is a straight line connecting those two points, so one must have strict inequality $m_\sigma(F) < m_\sigma(H_0(E))+m_\sigma(H_1(G))$ unless the HN curve of $H_1(G)$ is trivial, i.e., $G=0$.

So, we can recover the stable objects from the mass function $m_\sigma(E) = |e^{\ell_\sigma^t(E)}|$. The function $\phi_\sigma^t(F) = \frac{1}{\pi} \Im(\ell_\sigma^t(F))$ recovers the phases of stable objects, since by \Cref{P:firstproperties} we have $\phi_\sigma^t(F) = \phi_\sigma(F)$ for all (semi)stable $F$. The semistable objects are those that can be constructed by finitely many iterated extensions of stable objects of the same phase. Thus, we recover the slicing $\cP$ of $\sigma$. Finally, for any object $F$ we have $Z_\sigma(F) = \sum_j m_\sigma(F) e^{i \pi \phi_\sigma(G_j)}$, where $G_1,\ldots,G_n$ are the HN factors of $F$.    

By \eqref{E:translation}, $\ell_{z\cdot \sigma}^t(E) = \ell_\sigma^t(E) + z$ for any $z\in \bf{C}$. Thus, there is an induced continuous map $\overline{\ell}^t:\Stab(\cD)/\bf{C}\to \bf{C}^{\mathfrak{d}}/\bf{C}$. To see that $\overline{\ell}^t$ is injective, choose $F\in \mathfrak{d}$ and note that there is a section $s_F$ of $\Stab(\cD)\to \Stab(\cD)/\bf{C}$ given by sending $\overline{\sigma} \in \Stab(\cD)/\bf{C}$ to its unique representative $\tau$ with $\ell_\tau(F) = 1$. Then, $\overline{\ell}^t = \pi\circ \ell^t\circ s_F$ and the result follows from the fact that $\pi$ is injective on the subspace of $\bf{C}^{\mathfrak{d}}$ where the coordinate corresponding to $F$ equals one. Finally, since $\ell = \alpha \circ \d m$, for $\alpha: \mathscr{M}(\bf{R})^{\mathfrak{d}} \to \bf{C}^{\mathfrak{d}}$ as in \eqref{E:elldefinition}, it follows from injectivity of $\ell$ that $\d m$ is injective.
\end{proof}

\begin{rem}
    Composing $\ell:\Stab(\cD) \to \bf{C}^{\mathfrak{d}}$ with the projection $\bf{C}^{\mathfrak{d}}\xrightarrow{\ev_{E}} \bf{C}$ recovers the log central charge function $\sigma \mapsto \ell_\sigma(E)$. When $E$ is $\sigma$-semistable, $\ell_\sigma(E) = \logZ_\sigma(E) := \log m_\sigma(E) + i\pi \phi_\sigma(E)$. 

    Next, suppose $\sigma$ is a stability condition satisfying the support property with respect to a homomorphism $v:\rm{K}_0(\cD) \to \Lambda$, with $\Lambda$ a finite rank Abelian group. The strengthened version of Bridgeland's deformation theorem in \cite{Bayer_short} can be interpreted as saying that there exist an open neighborhood $U$ of $\sigma$ and a finite collection of objects $\mathfrak{e} = \{E_1,\ldots, E_n\}$, stable for all $\tau \in U$, such that $U\xrightarrow{\ell} \bf{C}^{\mathfrak{d}} \twoheadrightarrow \bf{C}^{\mathfrak{e}}$ is a homeomorphism onto its image.
\end{rem}

\begin{rem}
    Bridgeland \cite{Br07} equips $\Stab(\cD)$ with a canonical metric topology using \eqref{E:Brmetric}. In view of \Cref{T:measuretheoreticembedding}, one can instead equip $\Stab(\cD)$ with the subspace topology inherited from $\bf{C}^{\mathfrak{d}}$. We call this topology the \emph{weak} topology on $\Stab(\cD)$, since \Cref{T:measuretheoreticembedding} implies that it is weaker than the metric topology on $\Stab(\cD)$. Note that the weak topology is nevertheless Hausdorff, since the product topology on $\bf{C}^{\mathfrak{d}}$ is.
\end{rem}

\begin{prop}
\label{P:weakvsstrong}
    Let $\cD_\infty = \DCoh(\pt)^{\oplus \bf{Z}}$. The weak topology on $\Stab(\cD_\infty)$ is strictly weaker than the metric topology.
\end{prop}

\begin{proof}
    To establish the result, we need to produce a sequence that converges in the weak topology on $\Stab(\cD_\infty)$, but does not converge in the usual metric topology. For this, note that $\Stab(\cD_\infty) = \prod_{\bf{Z}} \Stab(\pt)$, so that we can specify a stability condition $\sigma \in \Stab(\cD_\infty)$ by a collection $(Z_n \in \bf{C} \cong \Hom(\bf{Z},\bf{C}))_{n\in \bf{Z}}$ of central charges on the heart $\rm{Vect}$. Define a sequence $(\sigma_{k})_{k\in \bf{Z}} \in \Stab(\cD_\infty)$ such that $(Z_{n,k})_{k\in \bf{Z}} \to 1$ pointwise, but not uniformly. Thus, the sequence $(\sigma_k)_{k\in \bf{Z}}$ converges in the weak topology, but not in the metric topology.
\end{proof}

The example of \Cref{P:weakvsstrong} is of a somewhat pathological nature. It would be interesting to determine under which circumstances the metric and weak topologies on $\Stab(\cD)$ coincide. 

\section{A Stronger embedding}
\label{S:strongerembedding}

In this section, we assume that stability conditions satisfy the support property with respect to a fixed map $v:\rm{K}_0(\cD)\to \Lambda$. 

Given a smooth real manifold $U$, a codimension one embedded submanifold $W\subset U$ is called a \emph{wall}. Given a collection $\cW = \{W_\alpha\}_{\alpha \in A}$ of walls, the set of $\cW$-chambers is the collection of connected components of 
\[
    U^\circ := U \setminus \bigcup_{\alpha \in A} W_\alpha.
\]
A $\cW$-stratum is any $\cW$-chamber or subspace of the form $\bigcap_{\beta \in B}W_\beta \setminus (\bigcup_{\gamma \in A\setminus B} W_\gamma)$ for $B\subset A$.

\begin{lem}
\label{L:wallandchamber}
    Let $\sigma \in \Stab(\cD)$ and $E\in \mathfrak{d}$ be given. There exists an open set $U$ containing $\sigma$ and a finite set of walls $\cW = \{W_\alpha\}_{\alpha \in A}$ passing through $\sigma$ such that the Harder--Narasimhan filtration of $E$ is constant on each $\cW$-stratum.
\end{lem}

\begin{proof}
    This is based on the ideas of \cite{BridgelandK3}*{Prop. 9.3}. Let $K$ be a compact and connected subset of $\Stab(\cD)$ whose interior contains $\sigma$. As explained in \emph{loc. cit.}, compactness of $K$ implies that $m_\tau(E)/m_\sigma(E)$ is uniformly bounded for all $\tau \in K$. This implies that the class $\cT$ of objects $A$ of $\cD$ such that $m_\tau(A)\le m_\tau(E)$ for some $\tau \in K$ has bounded mass in the sense of \cite{BridgelandK3}*{Def. 9.1}. In particular, $\mathfrak{H} = \{\HN_\bullet^\tau(E):\tau \in K\}$ has bounded mass.
    
    The support property implies that the conclusion of \cite{BridgelandK3}*{Prop. 9.3} holds in this context. Thus, there is a finite collection $\cW' = \{W_\alpha'\}_{\alpha \in A}$ of walls in $K$ defined by the relations $Z(T_1) \in \bf{R}_{>0}\cdot Z(T_2)$ for pairs of object $T_1,T_2 \in \cT$. Furthermore, for any $\cW'$-chamber $C$, if $F \in \mathfrak{H}$ is semistable for some $\tau \in C$, then it is semistable for all other elements of $C$.

    Now, choose an open set $U$ in the interior of $K$ containing $\sigma$. Up to shrinking $U$, we may assume that $W_\alpha' \cap U \ne \varnothing$ implies that $\sigma \in W_\alpha'$. We may also assume that each $W_\alpha := W_\alpha' \cap U$ is closed in $U$. Let $\cW = \{W_\alpha\}_{\alpha \in A}$, where we have abused notation by using $A$ again. We show that $U$ and $\cW = \{W_\alpha\}_{\alpha \in A}$ are as claimed in the statement. 
    
    First, observe that any $\cW$-chamber $C$ is contained in a $\cW'$-chamber in $K$. Now, choose $\tau \in C$ and let $\{\HN_i^\tau(E)\}\subset \mathfrak{H}$ denote the $\tau$-HN factors of $E$. As $\tau$ varies in $C$, each of the factors $\HN_i^\tau(E)$ remains semistable. Furthermore, since we cross none of the walls defined by the relations above, the inequalities $\phi(\HN_1^\tau(E)) > \cdots > \phi(\HN_m^\tau(E))$ are unchanged. Thus, HN filtrations are constant in each $\cW$-chamber.

    Finally, we consider HN filtrations on intersections of walls $S = \bigcap_{\beta \in B} W_\beta$. Since $S$ is in the closure of a chamber $C$, for each $\tau \in S$ we can find a sequence $(\tau_k)$ in $C$ such that $(\tau_k)\to \tau$. Then, by the argument of the preceding paragraph, $\{\HN_i^{\tau_k}(E)\}$ is constant in $k$. Thus, we denote the factors by $\{\HN_i^C(E)\}$. It follows from \cite{GKRpaper04}*{Thm. 4.1} that $\{E_{i-1}^C \to E_{i}^C\to \HN_i^C(E)\}_{i=1}^m$ is uniquely determined up to isomorphism on all of $C$. Since semistability of an object is a closed condition on $\Stab(\cD)$, all elements of $\{\HN_i^C(E)\}$ are $\tau$-semistable and 
    \[
        \phi_\tau(\HN_1^C(E)) \ge \cdots \ge \phi_\tau(\HN^C_m(E)).
    \]
    The $\tau$-HN-filtration of $E$ is obtained from $\{E_{i-1}^C\to E_i^C\to \HN_i^C(E)\}_{i=1}^m$ as follows: given a maximal sequence $i,\ldots, i+k$ of indices such that $\phi_\tau(\HN_i^C(E)) = \cdots = \phi_\tau(\HN_{i+k}^C(E)) =: \varphi$, the octahedral axiom implies that $\Cone(E_i\to E_{i+k})$ is an iterated extension $\{\HN_i^C(E),\ldots, \HN_{i+k}^C(E)\}$ and thus $\tau$-semistable of phase $\varphi$. Coarsening in this fashion determines the $\tau$-HN filtration of $E$ from the $C$-HN filtration and the set of coinciding phases on $S$. Thus, HN filtrations are constant on $S$.
\end{proof}

Next, fix $E\in \mathfrak{d}$ and consider the map $\Stab(\cD)\to \mathscr{M}(\bf{R})$ sending $E\mapsto \d m_{\sigma,E}$. The space $\mathscr{M}(\bf{R})$ admits many natural topologies, including the weak topology defined by $\mathscr{C}^0(\bf{R})$ as in \Cref{S:continuityproperties}. In terms of comparing convergence in $\mathscr{M}(\bf{R})^\mathfrak{d}$ to the metric topology on $\Stab(\cD)$, the weak topology has a major defect: the functions $\phi^+$ and $\phi^-$, with e.g. $\phi^{+}(\sum_{i}a_i \delta_{\theta_i}) = \max\{\theta_i\}$, are not continuous.

\begin{lem}
\label{L:upperlowernotcont}
    Equip $\mathscr{M}(\bf{R})$ with the weak topology dual to any class of functions $\cF$ on $\bf{R}$. The functions $\phi^{\pm}$ are not continuous.
\end{lem}

\begin{proof}
    Consider the sequence $\mu_n = \delta_0 + \tfrac{1}{n}\delta_1$ in $\mathscr{M}(\bf{R})$. We have $\mu_n\to \delta_0$ as $n\to \infty$. Indeed, taking any $f\in \mathcal{F}$ gives 
    \[
        \int_{\bf{R}} f\:\d\mu = f(0) + \tfrac{1}{n}f(1)\to f(0). 
    \]
    On the other hand, $\phi^+(\mu_n) = 1$ for all $n$, while $\phi^+(\delta_0) = 0$.
\end{proof}

The space $\mathscr{M}(\bf{R})$ admits a stronger topology coming from its structure as a limit of configuration spaces of points in $\bf{R}$, which we now explain. For any $d \ge 1$, let $\mathscr{M}_{\le d}(\bf{R})$ denote the set of measures $\mu$ with $\lvert \supp(\mu)\rvert \le d$. There is a canonical surjection 
\[
    \pi_d:\bf{R}^d\times \bf{R}^d_{>0}\twoheadrightarrow \mathscr{M}_{\le d}(\bf{R})
\]
given by $(\theta_1,\ldots, \theta_d,t_1,\ldots, t_d)\mapsto \sum_{i=1}^d t_i\cdot \delta_{\theta_i}$. We equip $\mathscr{M}_{\le d}(\bf{R})$ with the quotient topology induced by $\pi_d$. That is, a subset $U\subset \mathscr{M}_{\le d}(\bf{R})$ is declared open if and only if $\pi_d^{-1}(U)$ is open. There is a commutative diagram
\[
    \begin{tikzcd}
        \bf{R}^{d-1}\times \bf{R}^{d-1}_{>0} \arrow[r,two heads,"\pi_{d-1}"] \arrow[d,hook,"i_{d-1}",swap]& \mathscr{M}_{\le d-1}(\bf{R}) \arrow[d,hook]\\
        \bf{R}^{d}\times \bf{R}^{d}_{>0} \arrow[r, two heads,"\pi_d"]& \mathscr{M}_{\le d}(\bf{R})
    \end{tikzcd}
\]
where $i_{d-1}(\theta_1,\ldots, \theta_{d-1},t_1,\ldots, t_{d-1}) = (\theta_1,\ldots, \theta_{d-1},\theta_{d-1},t_1,\ldots, t_{d-2},\tfrac{1}{2}t_{d-1},\tfrac{1}{2}t_{d-1})$ and the other vertical map is inclusion. In particular, the canonical map $\mathscr{M}_{\le d-1}(\bf{R})\hookrightarrow \mathscr{M}_{\le d}(\bf{R})$ is continuous and in fact a homeomorphism onto its image. We topologize $\mathscr{M}(\bf{R})$ by the identity
\[
    \mathscr{M}(\bf{R}) = \varinjlim \mathscr{M}_{\le d}(\bf{R})
\] 
and refer to the resulting topology as the \emph{strong} topology on $\mathscr{M}(\bf{R})$. We have already defined $\phi^{\pm}$ for elements of $\mathscr{M}(\bf{R})$. We define the \emph{mass} of $\mu \in \mathscr{M}(\bf{R})$ by 
\[
    m(\mu) := \int_{\bf{R}}1\:\d\mu.
\]
In contrast to \Cref{L:upperlowernotcont}, we have: 

\begin{lem}
\label{L:continuousfunctionsstrong}
    With respect to the strong topology on $\mathscr{M}(\bf{R})$, the functions $\phi^{\pm}$ and $I_f(\mu) = \int_{\bf{R}} f\:\d\mu$ are continuous for any $f\in \mathscr{C}^0(\bf{R})$.
\end{lem}

\begin{proof}
    By definition of the topology on $\mathscr{M}(\bf{R})$, it suffices to check that these functions are cont\-inuous on $\mathscr{M}_{\le d}(\bf{R})$ for any $d\ge 1$. Then, by definition of the topology on $\mathscr{M}_{\le d}(\bf{R})$, it suffices to check after pulling back to $\bf{R}^d\times \bf{R}^d_{>0}$. For $\phi^+$, the corresponding function on $\bf{R}^d\times \bf{R}^d_{>0}$ is $\max\{\theta_i\}_{i=1}^n$, which is continuous. The argument for $\phi^-$ is analogous. For $f\in \mathscr{C}^0(\bf{R})$, the function $\int_{\bf{R}}f\:\d(-)$ pulls back to $(\theta_i,t_i)_{i=1}^d \mapsto \sum_{i=1}^d t_i\cdot f(\theta_i)$, which is continuous.
\end{proof}

\Cref{L:continuousfunctionsstrong} suggests that the strong topology on $\mathscr{M}(\bf{R})$ is suitable for comparison with that of $\Stab(\cD)$ --- henceforth, we consider $\mathscr{M}(\bf{R})$ with the strong topology unless otherwise specified. We define a new topology on $\mathscr{M}(\bf{R})^{\mathfrak{d}}$ which strengthens the product topology: consider the injection $\mathscr{M}(\bf{R})^{\mathfrak{d}} \hookrightarrow (\mathscr{M}(\bf{R}) \times \bf{R}^3)^{\mathfrak{d}}$ defined by 
\[ 
    (\mu_E)_{E\in \mathfrak{d}} \mapsto (\mu_E, \log m(\mu_E),\phi^+(\mu_E),\phi^-(\mu_E))_{E\in \mathfrak{d}}.
\]
Topologize $(\mathbf{R}^3)^{\mathfrak{d}}$ with the uniform topology: that is, convergence is defined using the generalized metric $d((x_E),(y_E)) = \sup_{E\in \mathfrak{d}}\{\lVert x_E-y_E\rVert\}$. We equip $(\mathscr{M}(\bf{R}) \times \bf{R}^3)^{\mathfrak{d}} \cong \mathscr{M}(\bf{R})^{\mathfrak{d}} \times (\bf{R}^3)^\mathfrak{d}$ with the product topology obtained by regarding it as a product of $\mathscr{M}(\bf{R})^{\mathfrak{d}}$ with $(\bf{R}^3)^{\mathfrak{d}}$, equipped with its uniform topology. Finally, the \emph{strong} topology on $\mathscr{M}(\bf{R})^{\mathfrak{d}}$ is the one it inherits as a subspace of $(\mathscr{M}(\bf{R}) \times \bf{R}^3)^{\mathfrak{d}}$, thus topologized.

By \Cref{L:continuousfunctionsstrong}, the strong topology on $\mathscr{M}(\bf{R})^{\mathfrak{d}}$ is at least as strong as the product topology; in fact, it is strictly stronger by \Cref{L:upperlowernotcont}, though we will not use this. 

We remind the reader that in this section, $\Stab(\cD)$ denotes the space of stability conditions satisfying the support property with respect to a fixed surjection $v:\rm{K}_0(\cD) \twoheadrightarrow \Lambda$. 

\begin{thm}
\label{T:strongembedding}
    The map $\d m:\Stab(\cD) \to \mathscr{M}(\bf{R})^{\mathfrak{d}}$ is an embedding when $\mathscr{M}(\bf{R})^{\mathfrak{d}}$ is equipped with the strong topology.
\end{thm}

\begin{proof}
    By \Cref{T:measuretheoreticembedding}, we know that $\d m:\Stab(\cD)\to \mathscr{M}(\bf{R})^{\mathfrak{d}}$ is injective. Consider a sequence $(\sigma_k)_{k\in \bf{N}}$ in $\Stab(\cD)$ such that $(\sigma_k)_{k\in \bf{N}} \to \sigma$ in the metric topology on $\Stab(\cD)$. We first show that $\d m_{\sigma_k,E} \to \d m_{\sigma,E}$ converges with respect to the strong topology on $\mathscr{M}(\bf{R})$. By \Cref{L:wallandchamber}, we may assume that $\sigma_k \in U$ in the notation of \emph{loc. cit.} for all $k$. Furthermore, we may assume that $\sigma_k$ lies in a $\cW$-stratum $S$ for all $k$.\endnote{We are using the following fact: Given a sequence $(\sigma_k)_{k\in \bf{N}}$ in a topological space $X$ which admits a partition into finitely many subsequences, each of which converges to a point $\sigma$, then $(\sigma)_{k\in \bf{N}}\to \sigma$ also. In our context, the subsequences are indexed by the $\cW$-strata in $U$.} So, there is a constant HN-filtration with factors $\{F_i\}_{i=1}^m$ and such that $\phi_{\sigma_k}(F_1) > \cdots > \phi_{\sigma_k}(F_m)$ for all $k$. Thus,
    \[
        \d m_{\sigma_k,E} = \sum_{i=1}^m m_{\sigma_k}(F_i) \cdot \delta_{\phi_{\sigma_k}(F_i)}.
    \]
    In the limit as $k\to \infty$, the inequalities of phases may become equalities. Assume for simplicity that for some $1\le p \le m$, $\phi_\sigma(F_{p}) = \cdots = \phi_\sigma(F_m)$ is a maximal chain of equalities. As explained in the proof of \Cref{L:wallandchamber}, the HN factor of $E$ for $\sigma$ with phase $\phi_\sigma(F_m)$, denoted $G$, is an extension of the objects $F_p,\ldots, F_m$, each appearing with multiplicity one. The equality 
    \[
        \lim_{k\to \infty} \sum_{i=p}^m m_{\sigma_k}(F_i) \cdot \delta_{\phi_{\sigma_k}(F_i)} = \sum_{i=p}^m m_\sigma(F_i) \cdot \delta_{\phi_\sigma(G)} = m_\sigma(G)\cdot \delta_{\phi_\sigma(G)}
    \]
    implies that $\lim_{k\to \infty} \d m_{\sigma_k,E} = \d m_{\sigma,E}$. The general situation is handled similarly. Thus, the map $\d m:\Stab(\cD) \to \mathscr{M}(\bf{R})^{\mathfrak{d}}$ is a continuous injection. Finally, we show that if $(\sigma_k)\to \sigma$ in the strong topology on $\mathscr{M}(\bf{R})^{\mathfrak{d}}$ then it converges in the metric topology on $\Stab(\cD)$. Let $\epsilon>0$ be given. By definition of convergence in the uniform topology on $(\bf{R}^3)^{\mathfrak{d}}$, there exists $k_0$ such that $k\ge k_0$ implies that 
    \[
        d(\sigma_k,\sigma) = \sup_{E\in \mathfrak{d}}\left\{\left\lvert \log \tfrac{m_{\sigma_k}(E)}{m_\sigma(E)} \right\rvert, \lvert \phi^{\pm}_{\sigma_k}(E) - \phi^{\pm}_\sigma(E) \rvert\right\}<\epsilon.
    \]
    Thus, $(\sigma_k)\to \sigma$ in the metric topology on $\Stab(\cD)$.
\end{proof}

In some special cases, the product topology on $\mathscr{M}(\bf{R})^{\mathfrak{d}}$, where each copy of $\mathscr{M}(\bf{R})$ is equipped with the strong topology, is already strong enough for $\d m:\Stab(\cD) \to \mathscr{M}(\bf{R})$ to be an embedding. Let $\cD$ be a triangulated category. We let
\begin{equation}
\label{E:stableobjectset}
    \mathfrak{s}(\cD) \:\::= \bigcup_{\sigma \in \Stab(\cD)} \{E\in \mathfrak{d}: E\text{ is }\sigma\text{-semistable}\}.
\end{equation}
We say that a triangulated category $\cD$ is of \emph{finite type} if $\mathfrak{s}(\cD)/[1]$ is a finite set. Given $\sigma,\tau \in \Stab(\cD)$, we define
\begin{equation}
\label{E:uniformDelta}
    \Delta(\sigma,\tau) := \sup_{S\in \mathfrak{s}}\left\{ \lvert \phi_{\sigma}^{\pm}(S) - \phi_\tau^{\pm}(S)\rvert, \left\lvert \log \frac{m_{\sigma}(S)}{m_\tau(S)}\right\rvert \right\}.
\end{equation}

\begin{prop}
\label{P:finitetypeconvergence}
    If $\cD$ is a finite type triangulated category, then a sequence $(\sigma_k)_{k\in\bf{N}}$ in $\Stab(\cD)$ converges to $\sigma$ with respect to Bridgeland's topology if and only if $\lim_{k\to\infty}\Delta(\sigma_k,\sigma) = 0$.
\end{prop}

\begin{proof}
    If $(\sigma_k)\to \sigma$, then the definition of the metric \eqref{E:Brmetric} implies that $\lim_{k\to\infty} \Delta(\sigma_k,\sigma) = 0$. Conversely, suppose that $\lim_{k\to\infty} \Delta(\sigma_k,\sigma) = 0$. Recall from \cite{Br07}*{Prop. 6.1} that 
    \begin{equation}
    \label{E:dsliceinf}
        d_{\rm{slice}}(\cP_{\sigma_k},\cP_\sigma) = \inf\left\{\epsilon\ge 0: \cP_\sigma(\phi)\subset \cP_{\sigma_k}[\phi-\epsilon,\phi+\epsilon] \text{ for all } \phi \in \bf{R} \right\}.
    \end{equation}
    However, we have  
    \begin{align*} 
        d_{\rm{slice}}(\cP_{\sigma_k},\cP_\sigma) & = \inf\left\{\epsilon\ge 0: \cP_\sigma(\phi)\subset \cP_{\sigma_k}[\phi-\epsilon,\phi+\epsilon] \text{ for all } \phi \in (0,1] \right\}\\
        & = \max\left\{ \lvert\phi_{\sigma_k}^{\pm}(E) - \phi_\sigma(E)\rvert: 0\ne E\in \cP_\sigma(0,1]\right\}
    \end{align*}
    where for the second equality we use the fact that there are finitely many semistable objects in $\cP_\sigma(0,1]$. Since $\lim_{k\to\infty}\Delta(\sigma_k,\sigma) = 0$, it follows that $\lim_{k\to\infty} \cP_{\sigma_k} = \cP_\sigma$ in $\Slice(\cD)$. 
    
    Because $\Stab(\cD)$ is equipped with the subspace topology as a subset of $\Hom(\Lambda,\bf{C})\times \Slice(\cD)$ \cite{Bayer_short}*{p. 1601}, it only remains to show that $\lim_{k\to\infty} Z_{\sigma_k} = Z_\sigma$. For this, choose $\sigma$-stable objects $S_1,\ldots,S_p$ that map to a generating set of $\Lambda$ under $v$. Then, by definition of $\Delta(\sigma_k,\sigma)$, we have that $\lim_{k\to\infty} \log m_\sigma(S_i)/m_{\sigma_k}(S_i) = 0$ for all $i=1,\ldots, p$. This, combined with the fact that $\lvert \phi_{\sigma_k}^+(S_i) - \phi_{\sigma_k}^-(S_i)\rvert \to 0$ for all $i=1,\ldots, p$ and \Cref{L:conevector} implies that $\max_{i=1}^p \left\{\lvert Z_{\sigma_k}(S_i) - Z_\sigma(S_i)\rvert\right\} \to 0$ and thus that $\lim_{k\to\infty} Z_{\sigma_k} = Z_\sigma$. 
\end{proof}

\begin{cor}
\label{C:finiteconvergence}
    Let $\mathcal{D}$ be a finite type triangulated category. The map $\d m:\Stab(\cD) \to \mathscr{M}(\bf{R})^{\mathfrak{d}}$ is an embedding when $\mathscr{M}(\bf{R})^{\mathfrak{d}}$ is equipped with the product topology.
\end{cor}

\begin{proof}
    By \Cref{C:continuousmaptomeasures}, we only need to show that if $(\d m_{\sigma_k})_{k\in \bf{N}}\to \d m_\sigma$ with respect to the product topology on $\mathscr{M}(\bf{R})^{\mathfrak{d}}$ then $(\sigma_k)_{k\in \bf{N}}\to \sigma$ in the metric topology on $\Stab(\cD)$. So, suppose that  $(\d m_{\sigma_k})_{k\in \bf{N}}\to \d m_{\sigma}$ in the product topology on $\mathscr{M}(\bf{R})^{\mathfrak{d}}$. Since $\mathfrak{s}$ is a finite set up to shift, by \Cref{L:continuousfunctionsstrong} we have $\lim_{k\to\infty} \Delta(\sigma_k,\sigma) = 0$. The result now follows from \Cref{P:finitetypeconvergence}.
\end{proof}

\begin{ex}
    Let $Q$ be a quiver of type ADE. In this case, $\DCoh(\rep\:Q)$ is of finite type in the above sense. Thus, \Cref{C:finiteconvergence} applies here. It seems likely that a similar result should hold for certain classes of ``discrete triangulated categories'' as considered in \cite{BPPDiscrete}.
\end{ex}

We make some concluding remarks about geometric structures on $\mathscr{M}_{\le d}(\bf{R})$. Its presentation as a quotient of $\bf{R}^n\times \bf{R}_{>0}^n$ gives it a canonical structure of a diffeological space in the sense of Souriau \cite{SouriauGroupes}. This endows $\mathscr{M}_{\le d}(\bf{R})$ with a sheaf of functions $\mathscr{C}^\infty$ declared ``smooth'' such that $\mathscr{C}^\infty(U):=\mathscr{C}^\infty(\pi_d^{-1}(U))$. The category of diffeological spaces is closed under colimits, so that $\mathscr{M}(\bf{R})$ has an induced diffeological space structure. 

This realizes $\mathscr{M}(\bf{R})$ as a space stratified by weighted configuration spaces. Indeed, for each $d\ge 1$, the locus $\mathscr{M}_d(\bf{R})$ consisting of measures $\mu$ with support of size $d$ is diffeomorphic to a trivial $\bf{R}_{>0}^d$-bundle over the non-orbifold locus of $\Sym^d(\bf{R}) = \bf{R}^d/\mathfrak{S}_d$.

\begin{lem}
    For all $f\in \mathscr{C}^\infty(\bf{R})$ the map $I_f(\mu) = \int_{\bf{R}}f\:\d\mu$ defines a smooth function on $\mathscr{M}(\bf{R})$. 
\end{lem}

\begin{proof}
    By definition, we must check that $I_f$ defines a smooth function on $\mathscr{M}_{\le d}(\bf{R})$ for all $d$. This is immediate, since $\pi_d^*(I_f)(\theta_1,\ldots, \theta_d,t_1,\ldots, t_d) = \sum t_i\cdot f(\theta_i)$, which is smooth whenever $f$ is. 
\end{proof}

In particular, taking $f(\theta) = \theta$ we see that the mass function $m:\mathscr{M}(\bf{R})\to \bf{R}$ is smooth in this sense. The category of diffeological spaces admits arbitrary products, so that $\mathscr{M}(\bf{R})^{\mathfrak{d}}$ admits a minimal diffeological structure such that the projections are all smooth. Consequently, $\Stab(\cD)$ admits a diffeological structure induced by the injection $\d m:\Stab(\cD) \hookrightarrow \mathscr{M}(\bf{R})^{\mathfrak{d}}$. 

\begin{rem}
    This smooth structure is distinct from the one given by Bridgeland's deformation theorem \cite{Br07}. Indeed, while the mass function of an object $E$ is a global smooth function on $\Stab(\cD)$ with the diffeological structure induced by $\d m$, mass functions are not smooth on $\Stab(\cD)$ with the usual smooth structure. 
    
    For example, consider $\DCoh(\rep A_2)$ and let $E = (\bf{C}\to \bf{C})$ be the non-trivial extension of $S_1$ by $S_2$. There is a short exact sequence $0\to S_2\to E\to S_1\to 0$. We can define a path $\sigma_t$ of stability conditions with heart $\rep A_2$ for $t\in (-\varepsilon,\varepsilon)$ by putting $Z_{\sigma_t}(S_2) = i$ and $Z_{\sigma_t}(S_1) = \sin(t) + i\cos(t)$. For $t\le 0$, $m_{t}(E) = \lvert Z_t(E)\rvert = \lvert Z_t(S_1)+ Z_t(S_2)\rvert$, while for $t\ge 0$, we have $m_t(E) = \lvert Z_t(S_1)\rvert + \lvert Z_t(S_2)\rvert$. A calculation shows that $m_t(E)$ fails to be smooth near $t=0$ by comparing second derivatives.\endnote{The two relevant functions are $z_t = (0,1)$ and $w_t = (\sin t,\cos t)$. We have $\lvert Z_t(S_1)\rvert + \lvert Z_t(S_2)\rvert = \lvert z_t\rvert + \lvert w_t\rvert = 2$ and $\lvert Z_t(E)\rvert = \lvert z_t+w_t\rvert = \sqrt{2}\cdot \sqrt{\cos t+1}$. All derivatives of $\lvert z_t\rvert + \lvert w_t\rvert$ are zero. On the other hand:
    \[
        \frac{d}{dt} \lvert Z_t(E)\rvert = \frac{1}{\sqrt{2}}\cdot \frac{-\sin t}{\sqrt{\cos t + 1}}
    \]
    and 
    \[
        \frac{d^2}{dt^2}\lvert Z_t(E)\rvert = \frac{1}{\sqrt{2}}\left(\frac{-\cos t\sqrt{\cos t + 1} -\tfrac{1}{2}\sin^2 t (\cos t+1)^{-1/2}}{\cos t +1}\right).
    \]
    Evaluating this expression at $t=0$ gives $-1/2$.} It would be interesting to see if one can use these two smooth structures to construct examples of non-equivalent smooth manifold structures on $\Stab(\cD)$. 
\end{rem}

\subsection*{Comparison with Thurston compactification}

As mentioned in the introduction, in recent years there have been many proposed enlargements of the space of stability conditions. Among these examples is the \emph{Thurston compact\-ification} of Bapat--Deopurkar--Licata \cite{ThurstonBDL}. This work gives a weak compactification of $\Stab(\cD)/\bf{C}$ for a broad class of triangulated categories by analogy to Thurston's work on compact\-ifying Teichm\"uller spaces. One chooses a class $S$ of objects in $\cD$ and studies, for $t\in \bf{R}$, the continuous maps 
\[
    m^t:\Stab(\cD)/\bf{C}\to \bf{P}_{\bf{R}}^S
\]
defined by $\sigma \mapsto [m_\sigma^t(\gamma):\gamma \in S]$. As pointed out in \cite{ThurstonKKO}, this map fails to be injective for $\cD = \DCoh(\bf{P}^1)$ for \emph{any} choice of $S$. However, in \cite{ThurstonBDL}*{Thm. 4.3} it is shown that for $n\ge 2$ and any $t_1<\cdots<t_n$ the map 
\[
    \prod_{i=1}^n m^{t_i} :\Stab(\cD)/\bf{C} \to (\bf{P}^S_{\bf{R}})^n
\]
is injective if $S$ contains a set of representatives of $\mathfrak{s}(\cD)$ as in \eqref{E:stableobjectset}. A remarkable fact is that if $S$ contains a finite set of objects whose sum is a classical generator of $\cD$, then the image of any $m^t$ is precompact by \cite{ThurstonBDL}*{Prop. 4.1}. Consequently, one can obtain from the framework of \emph{loc. cit.} a continuous injection $\Stab(\cD)/\bf{C}\hookrightarrow (\bf{P}^S_{\bf{R}})^n$ with compact closure. It remains to determine when these embeddings are homeomorphisms onto their image. For the Calabi--Yau-2 categories associated to the $\hat{A}_1$ and $A_2$ quivers, it is proven in \cite{ThurstonBDL} that $m = m^0$ is already injective, and in fact a homeomorphism onto its image. In \cite{DeopurkarThurstonk3}, an analogous result is proven in the case of $\Stab(S)$ for $S$ an analytic K3 surface with $\Pic(S) = 0$.

By contrast, the measure theoretic embeddings of \Cref{T:strongembedding} and \Cref{C:finiteconvergence} are homeo\-morphisms onto their images. There is a natural action of $\bf{C} = (\bf{C}^*)^{\sim}$ on $\mathscr{M}(\bf{R})$ given by regarding $w \in \bf{C}$ as $w = r+ i\pi \varphi$ for $r,\varphi \in \bf{R}$ and defining 
\[
    (r+i\pi\varphi) \cdot \sum_j a_j \cdot \delta_{\theta_j} =  \sum_j (e^r\cdot a_j) \cdot \delta_{\theta_j - \varphi}.
\]
One can verify that there is a commutative square 
\[
    \begin{tikzcd}
        \Stab(\cD) \arrow[d]\arrow[r,hook,"\d m"]& \mathscr{M}(\bf{R})^{\mathfrak{d}}\arrow[d]\\
        \Stab(\cD)/\bf{C}\arrow[r,hook,dashed,"\d \overline{m}"]& \mathscr{M}(\bf{R})^{\mathfrak{d}}/\bf{C}
    \end{tikzcd}
\]
where $\d\overline{m}$ is a homeomorphism whenever $\d m$ is. In particular, we obtain a homeomorphic emb\-edding of $\Stab(\cD)/\bf{C}$ into $\mathscr{M}(\bf{R})^{\mathfrak{d}}/\bf{C}$. However, this homeomorphic embedding typically cannot have pre-compact image, by contrast to \cite{ThurstonBDL}*{Prop. 4.1}. 

We say that a non-zero object $E$ is \emph{massless} relative to an object $F$ if there exists a sequence $(\sigma_k)_{k\in \bf{N}}$ such that $\lim_{k\to \infty} m_{\sigma_k}(E)/m_{\sigma_k}(F) = 0.$ We call the pair $(E,F)$ a pair of relatively massless objects for $(\sigma_k)_{k\in \bf{N}}$. Such pairs are a source non-compact directions in $\Stab(\cD)$.

\begin{lem}
\label{L:noncompact}
    If $\mathcal{D}$ admits a pair of relatively massless objects then the image of $\d \overline{m}$ is not pre-compact.
\end{lem}

\begin{proof}
    Consider a sequence of stability conditions $(\sigma_k)_{k\in \bf{N}}$ which admits a relatively massless pair $(E,F)$. If $(\d \overline{m}_{\sigma_k})_{k\in \bf{N}}$, or any of its subsequences, had a limit point $\overline{\mu}$, then there would exist a lift $\mu \in \mathscr{M}(\bf{R})^{\mathfrak{d}}$ such that $\int_{\bf{R}} \d\mu_E = 0$, which is impossible. 
\end{proof}

\begin{ex}
    It is an exercise to verify that the pair $(S_1,S_2)$ consisting of both simple represent\-ations of the $A_2$ quiver is relatively massless in $\DCoh(\rep\:A_2)$. Thus, by \Cref{L:noncompact} even in simple examples one cannot expect the image of $\d \overline{m}$ to be pre-compact. 
\end{ex}

\begin{rem}
    Many works in the literature address the issue of relatively massless objects, and construct enlargements $\Stab(\cD)$ or $\Stab(\cD)/\bf{C}$ where massless objects are allowed ---  cf. \cites{ThurstonBDL,Bolognesecompactification,laxstability}. In the case of \cite{augmented}, relatively massless objects sometimes give rise to \emph{non-admissible} boundary points. It should also be possible to add limit points to the corresponding sequences in the space of measures by allowing expressions of the form $\sum a_i \cdot \delta_{\theta_i}$ for $a_i \in \bf{R}_{\ge 0}$, but we will not pursue this here. 

    By \Cref{L:continuousfunctionsstrong}, for any pair of objects $E,F$ the function $\phi_\sigma(F) - \phi_\sigma(E)$ extends to continuous function on $\mathscr{M}(\bf{R})^{\mathfrak{d}}/\bf{C}$. Consequently, there are other non-compact directions arising from seq\-uences $(\sigma_k)_{k\in \bf{N}}$ for which $\lim_{k\to \infty} \lvert \phi_{\sigma_k}(E) - \phi_{\sigma_k}(F)\rvert = \infty$. Some well-behaved sequences (or paths) of this type, called \emph{quasi-convergent} in \cite{quasiconvergence}, have limit points given by admissible augmented stability conditions. 
\end{rem}

For any set of (isomorphism classes of) objects $\mathfrak{s}$, there is a natural continuous map $\mathscr{M}(\bf{R})^{\mathfrak{d}}/\bf{C} \to \mathscr{M}(\bf{R})^{\mathfrak{s}}/\bf{C}$ induced by projection. Furthermore, for each $t\in \bf{R}$ there is a continuous map
\[
    I_t:\mathscr{M}(\bf{R})^{\mathfrak{s}}/\bf{C}\to \bf{P}^\mathfrak{s}_{\bf{R}}
\]
defined by $(\mu_s)_{s\in \mathfrak{s}}\mapsto [\int_{\bf{R}}e^{t\theta}\:\d\mu_s(\theta):s\in \mathfrak{s}]$. Thus, for any choice of $t_1<\cdots<t_n$, we obtain a commutative diagram: 
\[
    \begin{tikzcd}
        \Stab(\cD)/\bf{C}\arrow[r,"\d \overline{m}^{\mathfrak{s}}"] \arrow[dr,"\prod m^{t_i}",swap]& \mathscr{M}(\bf{R})^{\mathfrak{s}}/\bf{C}\arrow[d,"\prod I_{t_i}"] \\
         & (\mathbf{P}^{\mathfrak{s}}_{\bf{R}})^n
    \end{tikzcd}
\]
so that integration carries the image of the measure theoretic embedding onto the image of the Thurston maps. If we choose $\mathfrak{s}$ such that $\d \overline{m}^{\mathfrak{s}}$ is a homeomorphism onto its image (possible by \Cref{T:strongembedding}) then $\prod m^{t_i}$ is a homeomorphism onto its image if and only if $\prod I_{t_i}$ restricts to a homeomorphism on the image of $\d \overline{m}^{\mathfrak{s}}$. In light of the discussion above, the following seems to be a reasonable intermediate step towards determining whether or not the Thurston map $\prod m^{t_i}$ is an embedding:

\begin{quest}
    Are the functions $\phi^+$ and $\phi^-$ continuous with respect to the topology on the image of $\prod m^{t_i}:\Stab(\cD)/\bf{C}\to (\bf{P}^{\mathfrak{s}}_{\bf{R}})^n$?
\end{quest}

\section{Triangle inequality for deformed mass functions}
\label{S:triangleinequality}

Throughout this section, we fix a pre-stability condition $\sigma$ on a $k$-linear triangulated category $\cD$ and $t\in \bf{R}$. For any $z \in \bC^*$, we let $g_t(z) := |z| e^{t \phi(z)}$, where $\phi(z) \in (-1,1]$ is the unique number such that $z = |z| e^{i \pi \phi(z)}$. We will consider sequences $(z_1,\ldots,z_n)$ ordered so that $\phi(z_1) \ge \cdots \ge \phi(z_n)$, and for any such sequence we let $g_t(z_1,\ldots,z_n) := \sum_{i} g_t(z_i)$. Given such a sequence, we let $\Sigma z_i := \sum_{j=1}^i z_i$ and $\Sigma z_0 := 0$. Note that one can subdivide a sequence
\[
(\ldots, z_{i-1},z_i,z_{i+1},\ldots) \mapsto (\ldots,z_{i-1},\lambda z_i, (1-\lambda) z_i,z_{i+1},\ldots)
\]
for any $\lambda \in (0,1)$ without changing the value of $g_t$. The function $g_t$ is also additive under concatenations of sequences, as long as $\phi(z_i)$ is still monotone decreasing for the concatenated sequence.

Our key lemma is the following, which removes the requirement in \cite{ikeda_2021}*{Lem. 3.7} that the $w_i,z_i \in \bH \cup \bR_{<0}$. See \Cref{fig:enter-label} for a visualization. 

\begin{lem}\label{L:gt_inequality}
Let $(z_1,\ldots,z_n)$ and $(w_1,\ldots,w_m)$ be sequences as above such that $\Sigma z_n = \Sigma w_m$. Let $r= \max\{j : \phi(z_j)>0\}$ and $s=\max\{j : \phi(w_j)>0\}$. Then $g_t(z_\bullet) - g_t(w_\bullet) \geq 0$ if both of the following conditions hold:
\begin{enumerate}
    \item $\forall\:i=1,\ldots,s$, $\Sigma w_i$ lies in the Minkowski sum $P_+ := \bR_{\geq 0} + [0,1] \cdot z_1 + \cdots + [0,1] \cdot z_r$, and 
    
    \item $\forall\:i=s,\ldots,m$, $\Sigma w_i$ lies in $P_- := \Sigma z_r + \bR_{\leq 0} + [0,1] \cdot z_{r+1} + \cdots + [0,1] \cdot z_n$. \vspace{-2mm}
\end{enumerate}
\end{lem}

\begin{figure}
    \centering
    \begin{tikzpicture}
        \coordinate (v1) at (-3,0.5);
        \coordinate (v2) at (0,-1);
        \draw (v2) edge[->] node[auto] {$w_1$} (-1.5,-0.5)
        	(-1.5,-0.5) edge[->] node[auto] {$w_2$} (-2.5,1.5) 
        	(-2.5,1.5) edge[->] node[auto] {$w_3$} (-1,4) 
        	(-1,4) edge[->] node[auto] {$w_4$} (1,3.5) 
        	(1,3.5) edge[->] node[auto] {$w_5$} (1.5,2.5) 
        	(1.5,2.5) edge[->] node[auto] {$w_6$} (0,1.5) 
        	(0,1.5) edge[->] node[auto] {$w_7$} (v1);
        
        \draw (1,1) edge[->] node[auto] {$z_9$} (v1)
        	(3,3) edge[->] node[auto] {$z_8$} (1,1)
        	(-0.5,5.5) edge[->] node[auto] {$z_7$} (3,3)
        	(-3,5) edge[->] node[auto] {$z_6$} (-0.5,5.5) 
        	(-4.5,4) edge[->] node[auto] {$z_5$} (-3,5)
        	(-5,2.5) edge[->] node[auto] {$z_4$} (-4.5,4)
        	(-3.5,-0.5) edge[->] node[auto] {$z_3$} (-5,2.5)
        	(-2,-1) edge[->] node[auto] {$z_2$} (-3.5,-0.5)
        	(v2) edge[->] node[auto] {$z_1$} (-2,-1);
    \end{tikzpicture}
    \caption{An example of the kind of curve to which \Cref{L:gt_inequality} applies. Conditions (1) and (2) of the lemma can be interpreted as follows: If $\gamma_z$ is the piecewise linear curve connecting $0,\Sigma z_1,\Sigma z_2,\ldots, \Sigma z_n$, and $\gamma_w$ is the analogous curve for $w_\bullet$, then the section of the curve $\gamma_z$ over which $\Im(-)$ is increasing (resp. decreasing) must lie to the left (resp. right) of the section of $\gamma_w$ where $\Im(-)$ is increasing (resp. decreasing).}
    \label{fig:enter-label}
\end{figure}
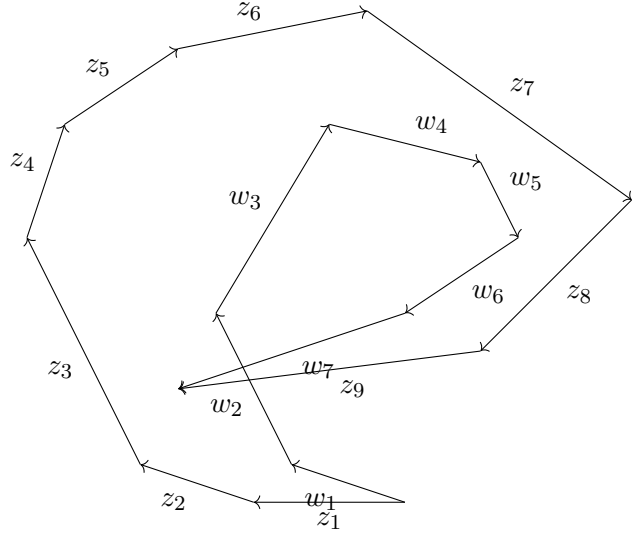

\begin{proof}
We will make use of \cite{ikeda_2021}*{Lem. 3.6}, which states that $g_t(u_1 + u_2) \leq g_t(u_1) + g_t(u_2)$ for any $u_1,u_2 \in \bH \cup \bR_{<0}$. However, if $u_1,u_2 \in \bC^*$ are any two points such that $\phi(u_2) \leq \phi(u_1) < \phi(u_2) + 1$, then we can multiply $u_1$ and $u_2$ by $e^{i\theta}$ for some $\theta \geq 0$ so that both lie in $\bH \cup \bR_{<0}$, and $g_t(u_1)$, $g_t(u_2)$, and $g_t(u_1+u_2)$ are all scaled by $e^{t\theta}$. We therefore have $g_t(u_1+u_2) \leq g_t(u_1) + g_t(u_2)$ in this more general setting, and we refer to this as the triangle inequality for $g_t$.

We will now prove \Cref{L:gt_inequality} by induction on $m+n$, proceeding in several cases.

\medskip
\noindent{\textit{Case 1, $m =0$:}} The claim is automatic, because $g_t(z_\bullet) \geq 0$ for any path. We assume $m\geq 1$ for the remaining cases.

\medskip
\noindent{\textit{Case 2, $\phi(z_1) = \phi(z_2)$}:} We simply replace $z_\bullet$ with $(z_1+z_2,z_3,\ldots)$, and the claim follows by the inductive hypothesis.

\medskip
\noindent{\textit{Case 3, $\phi(z_1)=\phi(w_1)$:}} We must have $w_1 = a z_1$ for some $a \in (0,1]$ in order for condition (1) to hold in case $\phi(w_1)>0$ or condition (2) to hold in case $\phi(w_1)\leq 0$. Therefore if $\phi(z_1)=\phi(w_1)$, we have $g_t(z_\bullet)-g_t(w_\bullet) = g_t((1-a)z_1,z_2,\ldots)-g_t(w_2,w_3,\ldots)$, and the claim follows by the inductive hypothesis, or by the $m=0$ case.

\medskip
\noindent{\textit{Simplifying assumptions:}}
\medskip

For the remainder of the proof, we may assume that cases 1, 2, and 3 do not occur. We now show that this implies the following:
\[
    \phi(w_1) < \phi(z_1) \text{ and } \phi(z_1)>\phi(z_2)>\phi(z_1)-1.
\]
In particular, $\phi(z_1) > \phi(z_1+z_2) > \phi(z_2) \geq \phi(z_3)$, and $z_1$ and $z_2$ span $\bC$ as a real vector space.

To see why $\phi(w_1) \leq \phi(z_1)$, and hence $\phi(w_1) < \phi(z_1)$, consider first the case where $\phi(w_1) > 0$. Condition (1) then implies $\phi(z_1)>0$ and $w_1 \in P_+ \subset \bR_{\geq 0} + R_{\geq 0} \cdot z_1$, and this is only possible if $\phi(w_1) \leq \phi(z_1)$. On the other hand, in the case where $\phi(w_1) \leq 0$ and $\phi(z_1)\leq 0$, condition (2) implies that $w \in P_- \subset \bR_{\leq 0}+\bR_{\geq 0} \cdot z_1$, and this also can only happen if $\phi(z_1) \geq \phi(w_1)$.

To see why $\phi(z_2)>\phi(z_1)-1$, assume instead that $\phi(z_2)\leq \phi(z_1)-1$. We must then have $\phi(z_1)>0$ and $\phi(z_2)\leq 0$. Condition (1) implies that $\Sigma w_s\in \bR_{>0}+[0,1] \cdot z_1$, where the strict inequality comes from our previous observation that $\phi(w_1) < \phi(z_1)$. This contradicts condition (2), because the latter implies that $\Sigma w_s \in P_- \subset z_1 + \bR_{\geq 0} \cdot z_2 + \bR_{\leq 0}$, and $(\bR_{>0}+[0,1] \cdot z_1) \cap (z_1 + \bR_{\geq 0} \cdot z_2 + \bR_{\leq 0}) = \varnothing$ if $\phi(z_2) \leq \phi(z_1)-1$.

\medskip
\noindent\textit{Case 4, $\phi(z_1+z_2) \leq \phi(w_1)$:} In this case, the ray $\bR_{\geq 0} \cdot w_1$ hits the interval $z_1 + [0,1] \cdot z_2$ at a unique point $z_1 + a z_2 = b w_1$ for some $a \in (0,1]$ and $b \geq 0$. We claim that $b \geq 1$: Indeed, if $\phi(z_2)>0$, then $b<1$ would imply that $w_1$ lies outside the region $\bR_{\geq 0} + [0,1] \cdot z_1 + \bR_{\geq 0} \cdot z_2$, which would violate condition (1). On the other hand, if $\phi(z_2) \leq 0$, then $b<1$ would imply that $w_1$ lies outside the region $z_1 + \bR_{\leq 0} + \bR_{\geq 0} \cdot z_2$ if $\phi(z_1) > 0$, or the region $\bR_{\leq 0}+[0,1]\cdot z_1 + \bR_{\geq 0}\cdot z_2$ if $\phi(z_1)\leq 0$, and in either case it would violate condition (2).

We then write
\begin{align*}
g_t(z_\bullet)-g_t(w_\bullet) &= [g_t(z_1,a z_2) - g_t(z_1+a z_2)] \\
&+ [g_t((b-1)w_1,(1-a) z_2,z_3,\ldots)-g_t(w_2,w_3,\ldots)].
\end{align*}
The first square bracketed term is nonnegative by the triangle inequality for $g_t$, and the second square bracketed term is nonnegative by either the $m=0$ case or the inductive hypothesis, once we have verified that the smaller curve satisfies conditions (1) and (2).

\begin{enumerate}[leftmargin=*,label=(\roman*)]
    \item \textit{Subcase $\phi(z_2)>0$:} For the new curve, $P_+$ is replaced by $-w_1 + P_+ \cap (\bR_{\geq 0} \cdot w_1 + \bR_{\geq 0})$, and $P_-$ is replaced by $-w_1+P_-$. At the same time $\Sigma w_i$ is replaced by $\Sigma w_i-w_1$ for all $i$. It is immediate that condition (2) holds of the new curve, and condition (1) holds for the new curve because in the original curve $\Sigma w_i \in w_1 + \bR_{\geq 0} \cdot w_1 + \bR_{\geq 0}$.
    
    \item \textit{Subcase $\phi(z_2)\leq 0$ and $\phi(w_1)\leq 0$:} In this case condition (1) is vacuous for both the old and new curves. $P_-$ for the old curve is replaced by $-w_1 + P_- \cap (\bR_{\geq 0}\cdot w_1 + \bR_{\leq 0})$, and $\Sigma w_i$ is replaced by $\Sigma w_i-w_1$, so condition (2) holds for the new curve.

    \item \textit{Subcase $\phi(z_2)\leq 0$ and $\phi(w_1)>0$:} In this case, $P_+ = \bR_{\geq 0} + [0,b-1] \cdot w_1$. Because $\phi(w_i)\leq \phi(w_1)$ for all $i$, condition (1) is equivalent to $\Im(\Sigma w_i) \leq \Im(b \cdot w_1)$ for $i=1,\ldots,s$. Because $\Im(\Sigma w_s)$ is maximal among $\Im(\Sigma w_i)$, this is equivalent to $\Im(\Sigma w_s) \leq \Im(b w_1)$. This holds because condition (1) for the original curve implies $\Sigma w_s \in P_- \cap ([0,b] \cdot w_1 + \bR_{\geq 0}) \subset (\bR \cdot w_1 + \bR_{\geq 0}) \cap (b w_1 + \bR \cdot z_2 + \bR_{\leq 0})$, and $b w_1$ is the unique maximizer for $\Im(z)$ in the latter region.
    
    Next, observe that $P_-$ for the old curve is replaced by $-w_1 + P_- \cap \{z \in \bC : \Im(z) \leq \Im(b w_1)\}$, and the $\Sigma w_i$ are replaced by $\Sigma w_i - w_1$. We have already observed that $\Im(\Sigma w_s)$ is maximal among $\Im(\Sigma w_i)$, and $\Im(\Sigma w_s) \leq \Im(b w_1)$ for the old curve. It follows that condition (2) holds for the new curve.
\end{enumerate}

\noindent\textit{Case 5, $\phi(z_1+z_2)>\phi(w_1)$:} In this case we write
\begin{align*}
g_t(z_\bullet)-g_t(w_\bullet) &= [g_t(z_1,z_2)-g_t(z_1+z_2)] \\
&+ [g_t(z_1+z_2,z_3,\ldots)- g_t(w_1,w_2,\ldots)].
\end{align*}
The first bracketed term is nonnegative by the triangle inequality for $g_t$. The second bracketed term will be nonnegative by the inductive hypothesis, and the proof complete, once we verify that the paths $(z_1+z_2,z_3,\ldots)$ and $(w_1,w_2,\ldots)$ still satisfy the conditions (1) and (2). We do this in three cases:

\begin{enumerate}[leftmargin=*,label=(\roman*)]
    \item \textit{Subcase $\phi(z_2) > 0$:} $P_-$ for the new curve is unchanged, so condition (2) continues to hold. $P_+$ is replaced by the closure of $P_+ \setminus \Delta$, where $\Delta$ is the triangle with vertices $0,z_1,z_1+z_2$. Because $\phi(\Sigma w_1) \geq \cdots \geq \phi(\Sigma w_s)$ and $\phi(\Sigma w_1) = \phi(w_1) < \phi(z_1+z_2)$, none of these points lie in $\Delta$, so the condition (1) holds.

    \item \textit{Subcase $\phi(z_1+z_2) \leq 0$:} Here $\phi(w_i)<0$ for all $i$. The condition (1) is vacuous, because $\phi(\Sigma w_i) < 0$ for all $i \geq 1$ as well. The condition (2) only depends on $P_- \cap \{z \in \bC:\Im(z) \leq 0\}$, and for the new curve this region is replaced with the closure of $(P_- \setminus \Delta) \cap \{z \in \bC:\Im(z) \leq 0\}$. Because $\phi(\Sigma w_i) \leq \phi(w_1) < \phi(z_1+z_2)$ for all $i$, none of the $\Sigma w_i$ lies in $\Delta$, so condition (2) continues to hold.

    \item \textit{Subcase $\phi(z_1+z_2) > 0 \geq \phi(z_2)$:} Because $\phi(\Sigma w_i) < \phi(z_1+z_2)$ for $i=1,\ldots,s$, the condition (1) for the new curve is equivalent to requiring that $\Im(\Sigma w_i)\leq \Im(z_1+z_2)$. Because $\Im(\Sigma w_s)$ is maximal among all $\Im(\Sigma w_i)$, this is equivalent to $\Im(\Sigma w_s) \leq \Im(z_1+z_2)$. But condition (2) implies that $\Sigma w_s$ lies in both the region $z_1+\bR_{\geq 0} z_2 + \bR_{\leq 0}$, which combined with $\phi(\Sigma w_s) < \phi(z_1+z_2)$ implies that $\Im(\Sigma w_s) \leq \Im(z_1+z_2)$.

    We have seen that $\Im(\Sigma w_i) \leq \Im(\Sigma w_s) \leq \Im(z_1+z_2)$ for all $i$. So condition (2) for the new curve holds, since $P_- \cap \{z \in \bC : \Im(z) \leq \Im(z_1+z_2)\}$ is the same for the new curve and the original curve. \qedhere
\end{enumerate}
\end{proof}

\begin{lem}[Triangle inequality]\label{L:triangle_inequality}
For any exact triangle $E \to F \to G$ in $\cD$:
\begin{equation}
\label{E:triangle_inequality}
    0\le m_\sigma^t(E) + m_\sigma^t(G) - m_\sigma^t(F). 
\end{equation}

\end{lem}

\begin{proof}
The gap in the proof of \cite{ikeda_2021}*{Prop. 3.3} occurs in the proof of \cite{ikeda_2021}*{Lem. 3.8},\endnote{\cite{ikeda_2021}*{Lem. 3.7} is only formulated and proven for $z_i$ in the upper half-plane, but the proof of \cite{ikeda_2021}*{Lem. 3.8} invokes the previous lemma in situations where $z_i$ is not in the upper half plane.} which establishes \eqref{E:triangle_inequality} for an exact triangle $C[-1] \to A \to B$ coming from a short exact sequence $0 \to A \to B \to C \to 0$ in the heart $\cP_\sigma(0,1]$. Here we will just repair the proof of that lemma.

Let $M$, $N$, and $P$ be the length of the HN filtration of $A$, $B$, and $C$, respectively. We let $w_i := Z_\sigma(\HN_i(A))$ for $i=1,\ldots, M$ and
\[
z_i := 
    \left \{ \begin{array}{ll} Z_\sigma(\HN_i(B)),    & \text{if } i=1,\ldots,N \\
    - Z_\sigma(\HN_{i-N}(C)), & \text{if } i=N+1,\ldots,N+P. \end{array} \right. 
\]
Then the tuples $(z_1,\ldots,z_{N+P})$ and $(w_1,\ldots,w_M)$ satisfy the hypotheses of \Cref{L:gt_inequality}, and thus
\[
0 \leq g_t(z_\bullet) - g_t(w_\bullet) = m^t_\sigma(B)+m^t_\sigma(C[-1]) - m^t_\sigma(A). \qedhere
\] 
\end{proof}

\subsection{Reverse triangle inequality}

Next, we prove a ``reverse'' version of \Cref{L:triangle_inequality}. Recall that associated to a slicing $\cP$ and a real number $a \in \bf{R}$ there is an associated pair of (homological) t-structures: $(\cP(>a),\cP(\le a))$ and $(\cP(\ge a),\cP(<a))$. In particular, there are projection functors 
\[
    \tau^{\star a}: \cD \to \cP(\star a) 
\]
where $\star\in \{\le, \ge, <,>\}$. Given an object $E$ of $\cD$, we write $\tau^{\star a}(E) = E^{\star a}$ for each such $\star$. We also write $\cD^{\star a} = \cP(\star a)$. In what follows, for a morphism $A\to B$ we let $\fib(A\to B) = \Cone(A\to B)[-1]$.

\begin{lem}\label{L:mass_additivity}
Let $E \to F \to G$ be an exact triangle in a triangulated category $\cD$ equipped with a pre-stability condition $\sigma$. Then for any $a, t\in \bf{R}$ one has
\[
    m^t(\Cone(E^{>a} \to F^{>a})) + m^t(\fib(F^{\leq a} \to G^{\leq a})) \leq m^t(E^{\leq a}) + m^t(G^{>a}).
\]
\end{lem}
\begin{proof}
Define $E_1 := \fib(F \to G^{\leq a})$. The octahedral axiom applied to the composition $F \to G \to G^{\leq a}$ gives an exact triangle $E \to E_1 \to G^{>a}$. Next, the octahedral axiom applied to the composition $E_1^{>a} \to E_1 \to F$ gives an exact triangle
\[
    E_1^{\leq a} \to \Cone(E_1^{>a} \to F) \to G^{\leq a}. 
\]
This implies that $\Cone(E_1^{>a} \to F) \in \cC^{\leq a}$, and thus $E_1^{>a} = F^{>a}$ and $E_1^{\leq a} = \fib(F^{\leq a} \to G^{\leq a})$. Now, let $N:= \Cone(E^{>a} \to E_1)$. The octahedral axiom applied to the composition $E^{>a} \to E \to E_1$ gives an exact triangle
\begin{equation}\label{E:exact_triangle_for_mass_bound}
    E^{\leq a} \to N \to G^{>a}.
\end{equation}
On the other hand, the octahedral axiom applied to the composition $E^{>a} \to F^{>a}=E_1^{>a} \to E_1$ gives an exact triangle
\[
    \Cone(E^{>a} \to F^{>a}) \to N \to E_1^{\leq a}.
\]
Note that $\Cone(E^{>a} \to F^{>a}) \in \cC^{>a}$, so this second triangle agrees with the triangle $N^{>a} \to N \to N^{\leq a}$. In particular,
\[
    m^t(N) = m^t(\Cone(E^{>a} \to F^{>a})) + m^t(\fib(F^{\leq a} \to G^{\leq a})).
\]
The claim now follows from the triangle inequality applied to \eqref{E:exact_triangle_for_mass_bound}.
\end{proof}

\begin{cor}[Reverse triangle inequality]\label{C:reverse_triangle_inequality}
For any $a \in \bR$ and any exact triangle $E \to F \to G$ in $\cD$ one has
\[
m^t(E) + m^t(G) \leq m^t(F) + (1+e^{|t|}) \left(m^t(E^{\leq a}) +m^t(G^{>a}) \right).
\]
\end{cor}
\begin{proof}
Applying \Cref{L:triangle_inequality} to the triangles $\Cone(E^{>a} \to F^{>a}) [-1] \to E^{>a} \to F^{>a}$ and $F^{\leq a} \to G^{\leq a} \to \fib(F^{\leq a} \to G^{\leq a})[1]$ gives
\[
m^t(F) \geq m^t(G^{\leq a}) + m^t(E^{>a}) - e^{-t} m^t(\Cone(E^{>a} \to F^{>a})) - e^t \fib(F^{\leq a} \to G^{\leq a}).
\]
Note that here we have used the identity $m^t(E[n]) = e^{tn}m^t(E)$ for all $n\in \bf{Z}$. In the case where $t \geq 0$, we apply \Cref{L:mass_additivity} to conclude
\begin{align*}
    m^t(F) &\geq m^t(G^{\leq a}) + m^t(E^{>a}) -e^{-t}(m^t(G^{>a})+m^t(E^{\leq a})) - (e^t-e^{-t}) m^t(\fib(F^{\leq a} \to G^{\leq a})) \\
    &\geq m^t(G^{\leq a}) + m^t(E^{>a}) -e^{t}(m^t(G^{>a})+m^t(E^{\leq a}))\\
    &=m^t(G)+m^t(E) - (1+e^t) (m^t(G^{>a})+m^t(E^{\leq a}))
\end{align*}
where the second inequality comes from the inequality $m^t(\fib(F^{\leq a} \to G^{\leq a})) \leq m^t(G^{>a})+m^t(E^{\leq a})$, and the fact that $e^t-e^{-t} \geq 0$ if $t \geq 0$. The $t\leq 0$ case is proved similarly.\endnote{The chain of inequalities is
\begin{align*}
m^t(F) &\geq m^t(G^{\leq a}) + m^t(E^{>a}) -(e^{-t}-e^{t})m^t(\Cone(F^{>a} \to G^{>a})) - e^{t}(m^t(G^{>a})+m^t(E^{\leq a}))\\
&\geq m^t(G^{\leq a}) + m^t(E^{>a}) -e^{-t}(m^t(G^{>a})+m^t(E^{\leq a}))\\
&=m^t(G)+m^t(E) - (1+e^{-t}) (m^t(G^{>a})+m^t(E^{\leq a})).
\end{align*}
This time we used that $e^{-t}-e^t \geq 0$ for $t\leq 0$.
}
\end{proof}

\section{Estimates for truncations}
\label{S:estimatesfortruncations}

This section is dedicated to proving the following:

\begin{thm} \label{T:filtration_inequality}
    Let $a_1<\cdots<a_{n-1}$, and let $\epsilon \in (0,1)$. Then there is a constant $C_{n,\epsilon,t}>0$, depending only on $n$, $\epsilon$, and $t$, such that for any $E \in \cD$ with a finite descending filtration $E = E_n \to E_{n-1} \to \cdots \to E_1 \to E_0 = 0$ with associated graded objects $F_i := \fib(E_i \to E_{i-1})$,
    \[
        m^t(E) \leq \sum_{i=1}^n m^t(F_i) \leq m^t(E) + C_{n,\epsilon, t} \left( \sum_{j=1}^{n-1} m^t(F_{j+1}^{\leq a_j}) + m^t(F_j^{>a_j-\epsilon}) \right).
    \]
\end{thm}

As a first step towards the proof, we will first establish estimates for the mass of the \emph{truncations} of objects in an exact triangle.

\begin{lem}\label{L:bound_gt_image}
Let $\delta : G \to E$ be a morphism in the heart $\cP_\sigma(0,1]$. Then for any $t \in \bf{R}$ and any $a,b \in (0,1)$, one has
\[
    g_t(Z(\im(\delta))) \leq \max \left\{ \begin{array}{c} m^t(G^{\leq a}) + \frac{e^{ta}}{\sin(\pi a)} \min\left( \Im(Z(G^{>a})),\Im(Z(E)) \right), \\ m^t(E^{>b}) + \frac{e^{tb}}{\sin(\pi b)} \min\left(\Im(Z(E^{\leq b})),\Im(Z(G))\right) \end{array} \right\}.
\]
\end{lem}

Note that when $a$ is close to $1$, the top expression simplifies to $m^t(G^{\leq a})$, which is $\leq m^t(G)$, and when $b$ is close to $0$, the bottom expression simplifies to $m^t(E)$.\endnote{As $a\to 1$ there is at most one HN factor of $G$ with phase in $(a,1]$, i.e. $G_1\in \cP(1)$, if it exists. However, in that case for $a$ sufficiently large, $G^{>a} = G_1$ and $\Im Z(G_1) = \Im Z(G^{>a}) = 0$.}

\begin{proof}
There are two constraints on the value of $Z(\im(\delta))$ that bound it in a compact polygon. For any $X \in \cP(0,1]$, let $\HN(X) \subset \bH \cup \bR_{<0}$ denote the HN polygon of $X$ \cite{Bayer_short}*{Def. 3.1}; i.e. the convex hull of the central charges $Z(A)$ of all subobjects $A \subset X$. By definition, $\im(\delta) \subset E$ implies $Z(\im(\delta)) \in \HN(E)$. On the other hand $\ker(\delta) \subset G$ implies $Z(\ker(\delta)) = Z(G) - Z(\im(\delta)) \in \HN(G)$, or equivalently $Z(\im(\delta)) \in Z(G) - \HN(G)$. Therefore
\[
Z(\im(\delta)) \in P:= \HN(E) \cap (Z(G)-\HN(G)),
\]
which is compact: It is bounded on the left by the HN curve of $E$. It is bounded on the right by the curve obtained as the transform $Z(G)-(-)$ applied to the HN curve of $G$. Finally, its points have imaginary part in $[0,h]$, where $h$ is defined to be the minimum of $\Im(Z(G)), \Im(Z(E))$, and the imaginary part of the intersection point in $\bH$ of the first two curves, if one exists.

Let $\gamma_E$ denote the curve obtained by concatenating the HN curve of $E^{>b}$ with the curve $Z(E^{>b}) + \bR_{\geq 0} e^{i \pi b}$, and let $\gamma_G$ be the concatenation of $Z(G^{\leq a})-$(HN curve of $G^{\leq a}$) with the curve $Z(G^{\leq a}) + \bR_{\geq 0} e^{i \pi a}$. Let $P' \supseteq P$ denote the larger polytope that is bounded on the left by $\gamma_E$ and above by the line $\Im(z)=h$, and on the right by $\gamma_G$. Let $u,v \in \bH \cup \bR_{\leq 0}$ denote the intersection of $\gamma_E$ and $\gamma_G$ with the line $\Im(z)=h$ respectively. Note that $\Re(u) \leq \Re(v)$. We write $P' = A \cup B \cup C$ as a union of polygons, where $B$ is the triangle with vertices $0,u,v$, $A$ is the region to the left of the line segment from $0$ to $u$, and $C$ is the region to the right of the line segment from $0$ to $v$. See \Cref{fig:truncated_triangle}. 

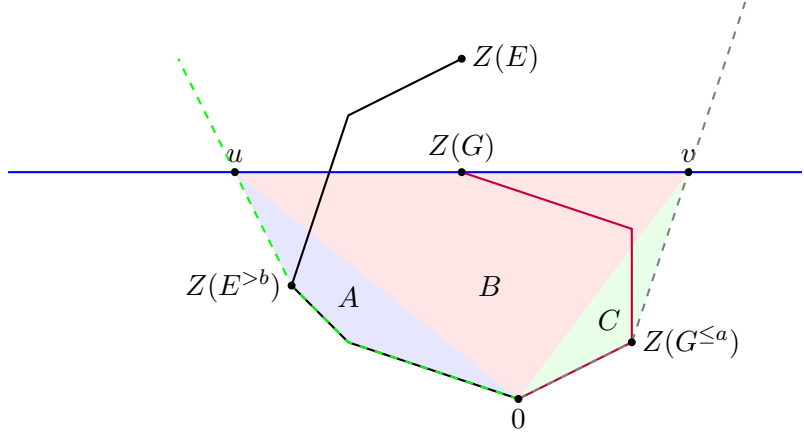
\begin{figure}
\begin{tikzpicture}[scale=1.50]
\coordinate (v1) at (0,-0.5);
\coordinate (v2) at (-1.5,0);
\coordinate (v3) at (-2,0.5);
\coordinate (v4) at (-1.5,2);
\coordinate (v5) at (-0.5,2.5);
\coordinate (v6) at (-3,2.5) {};
\coordinate (v9) at (-0.5,1.5);
\coordinate (v8) at (1,1);
\coordinate (v7) at (1,0);
\coordinate (v10) at (-4.5,1.5);
\coordinate (v11) at (2.5,1.5);
\coordinate (v12) at (-1.666666,1.5);
\coordinate (v13) at (-2.5,1.5) {};
\coordinate (v14) at (1.5,1.5);
\coordinate (v15) at (2,3);

\fill[red, opacity=.1] (v1)--(v13)--(v14)--cycle;
\fill[blue,opacity=.1] (v1) -- (v2) -- (v3) -- (v13)--cycle;
\fill[green,opacity=.1] (v1) -- (v14) --(v7)--cycle;

\draw[thick,blue]  (v10) edge (v11);
\draw[thick,purple] (v1) -- (v7) -- (v8) -- (v9);
\draw[thick] (v1) -- (v2) -- (v3) -- (v4) -- (v5);
\draw[thick,dashed,green] (v1) -- (v2) -- (v3) -- (v6);
\draw[thick,dashed,gray] (v1) -- (v7) -- (1.5,1.5) -- (v15);

\node[anchor=south] at (v13) {$u$};
\node[anchor=south] at (v14) {$v$};
\node[anchor=north] at (v1) {$0$};
\node at (-.25,0.5) {$B$};
\node at (0.8,0.2) {$C$};
\node at (-1.5,0.4) {$A$};
\node[anchor=south] at (v9) {$Z(G)$};
\node[anchor=west] at (v5) {$Z(E)$};
\node[anchor=east] at (v3) {$Z(E^{>b})$};
\node[anchor=west] at (v7) {$Z(G^{\leq a})$};
\fill[black] (v13) circle (1pt) (v14) circle (1pt) (v5) circle (1pt) (v9) circle (1pt) (v1) circle (1pt) (v3) circle (1pt) (v7) circle (1pt);
\end{tikzpicture}
\caption{This illustrates the polygons appearing in the proof of \Cref{L:bound_gt_image}. The blue horizontal line is the curve $\Im(z)=h$, the green dashed line is the curve $\gamma_E$, the black line is the HN curve of $E$, the purple line is $Z(G)-$(HN curve of $G$), and the gray dashed line is $\gamma_G$. $Z(\im(\delta))$ is constrained to lie in the polygon $P$ bounded by the blue, black, and purple curves. For the purposes of our bound, though, we maximize $g_t(z)$ on the larger polygon $P' = A \cup B \cup C$.}
\label{fig:truncated_triangle}
\end{figure}

We wish to upper bound the function $g_t(z)$ on $P'$, and it suffices to do this separately on each polygon $A,B,$ and $C$. Because $g_t(z)$ is monotone increasing in $|z|$, the maximum on $B$ must occur on the segment between $u$ and $v$, and in fact the maximum must occur at either $u$ or $v$.\endnote{One can compute using polar coordinates that $\frac{\partial}{\partial x} g_t(x+ih) = \frac{e^{t \phi} }{|x+ih|} (x- \frac{t}{\pi}h)$, so the only critical point of $g_t(z)$ along the line $\Im(z)=h$ is the point $z = th/\pi + i h$, and this is a local minimum.} Therefore, the maximum on $P'$ agrees with the maximum of $g_t(z)$ on $A \cup C$. Monotonicity in $|z|$ also implies that the maximum on $A$ and $C$ must occur on the curve $\gamma_E$ and $\gamma_G$ respectively.

For any $w \in \gamma_E \cap A$, one can apply \Cref{L:gt_inequality} with $w_\bullet = (w)$ and $z_\bullet$ consisting segments of $\gamma_E$ up to the point $w$ to conclude that $g_t(w) \leq g_t(z_\bullet)$. If we let $(z'_\bullet)$ consist of segments of $\gamma_E$ up to $u$, then we have shown that
\[
\sup_{z \in A} g_t(z) \leq g_t(z'_\bullet) = m^t(E^{>b}) + e^{tb} \max\left(0,\frac{h-\Im(Z(E^{>b}))}{\sin(\pi b)} \right),
\]
where the last term on the right is $g_t(u-Z(E^{>b}))$ if $h>\Im(Z(E^{>b}))$ and zero otherwise. We simplify this bound by noting both that $h-\Im(Z(E^{>b})) \leq h \leq \Im(Z(G))$ and $h-\Im(Z(E^{>b})) \leq \Im(Z(E)-Z(E^{>b})) = \Im(Z(E^{\leq b}))$.

One can use the same analysis on the polygon $C$ to show that\endnote{In this case, one must find an upper bound on $g_t(w)$ for $w \in \gamma_G \cap P'$. In order to do this, one applies \Cref{L:gt_inequality} where $w_\bullet = (w)$, and $z_\bullet$ are the segments of $\gamma_G$ connecting $0$ to $w$, reordered in such a way that $\phi(z_1)\geq \phi(z_2) \geq \cdots$. Note that because all segments of $\gamma_G$ lie in $\bH$ and $m=1$, the hypotheses of \Cref{L:gt_inequality} are automatically satisfied.}
\[
\sup_{z \in C} g_t(z) \leq m^t(G^{\leq a}) + \frac{e^{ta}}{\sin(\pi a)} \min(\Im(Z(G^{>a})),\Im(Z(E))).
\]
Combining these bounds for $g_t(z)$ on $A$ and $C$ gives the statement of the lemma.
\end{proof}

\begin{lem} \label{L:pre_triangle_inequality}
Let $E \to F \to G$ be an exact triangle in $\cD$. Consider the induced morphism $\delta : G^{(0,1]} \to E^{(-1,0]}[1]$ in the heart $\cP_\sigma(0,1]$ and let $\im(\delta)$ be its image. Then one has
\begin{align*}
    m_t(F^{>0}) &\leq m^t(G^{>0}) + m^t(E^{>0}) + e^{-t} g_t(Z(\im \delta)), \text{ and} \\
    m_t(F^{\leq 0}) &\leq m^t(G^{\leq 0}) + m^t(E^{\leq 0}) +  g_t(Z(\im \delta)).
\end{align*}
\end{lem}

\begin{proof}
Consider the $t$-structure associated to $\cP_\sigma(0,1]$, let $G_0 := G^{(0,1]}$ and $E_{-1}:= E^{(-1,0]}[1]$. The long exact homology sequence implies\endnote{Let $C=\Cone(E^{>0} \to F^{>0})$. Then the map of exact triangles induces a map of long exact sequences
\[
    \xymatrix{ \cdots \ar[r] & C_1 \ar[r] \ar[d]^{\cong} & E_0 \ar[r] \ar[d]^{=} & F_0 \ar[r] \ar[d]^{=} & C_0 \ar[r] \ar[d] & 0 & \cdots \\
    \cdots \ar[r] & G_1 \ar[r] & E_0 \ar[r] & F_0 \ar[r] & G_0 \ar[r]^\delta & E_{-1} \ar[r] & \cdots }.
\]
This induces an isomorphism $C_0 \xrightarrow{\sim} \ker(\delta)$, and all negative homology objects of $C$ vanish.}
\[
m^t(\Cone(E^{>0} \to F^{>0})) = m^t(G^{>1})+m^t(\ker(\delta)).
\]
Let $w_i = Z(\gr^{\HN}_i(\ker(\delta)))$, $z_i = Z(\gr^{\HN}_i(G_0))$ for $i=1,\ldots,N$, and $z_{N+1}=-Z(G_0/\ker(\delta)) = -Z(\im(\delta))$, where $N$ is the length of the HN filtration of $G_0$. Applying \Cref{L:gt_inequality} to $z_\bullet$ and $w_\bullet$ implies that $m^t(\ker(\delta)) \leq m^t(G_0) + g_t(-Z(\im(\delta)))$. Combining this with the triangle inequality $m^t(F^{>0}) \leq m^t(E^{>0}) + m^t(\Cone(E^{>0} \to F^{>0}))$ gives the first inequality in the lemma.

For the second inequality, the long exact homology sequence implies that\endnote{If we let $K = \fib(F^{\leq 0} \to G^{\leq 0})$, then we get a map of long exact homology sequences
\[
\xymatrix{ \cdots \ar[r] & G_0 \ar[r]^\delta \ar[d] & E_{-1} \ar[r] \ar[d] & F_{-1} \ar[r] \ar[d]^{=} & G_{-1} \ar[r] \ar[d]^{=} & E_{-2} \ar[r] \ar[d]^{\cong} & \cdots \\
\cdots \ar[r] & 0 \ar[r] & K_{-1} \ar[r] & F_{-1} \ar[r] & G_{-1} \ar[r] & K_{-2} \ar[r] & \cdots }.
\]
This shows that $K_{-1} \cong \coker(\delta)$.}
\[
m^t(\fib(F^{\leq 0} \to G^{\leq 0})) = m^t(E^{\leq -1}) + e^{-t} m^t(\coker(\delta)).
\]
Let $w_i = -Z(\gr_i^{\HN}(\coker(\delta)))$, $z_i = -Z(\gr_i^{\HN}(E_0))$ for $i=1,\ldots,N$, and $z_{N+1} = Z(\im(\delta))$, where $N$ is the length of the HN filtration of $E_{-1}$. Applying \Cref{L:gt_inequality} to $(z_\bullet)$ and $(w_\bullet)$ shows that $e^{-t} m^t(\coker(\delta)) \leq e^{-t} m^t(E_{-1}) + g_t(Z(\im(\delta)))$. This combined with the triangle inequality $m^t(F^{\leq 0}) \leq m^t(\fib(F^{\leq 0} \to G^{\leq 0})) + m^t(G^{\leq 0})$ gives the second inequality of the lemma.
\end{proof}

Combining the previous two lemmas yields:

\begin{prop}[Triangle inequality with truncation] \label{P:truncated_triangle_inequality}
Let $E \to F \to G$ be an exact triangle in $\cD$. Then for any $t,a \in \bR$ and $0<\epsilon<1$,
\begin{align*}
m^t(F^{>a}) &\leq m^t(E^{>a-\epsilon}) + m^t(G^{>a}) + r_\epsilon(-t) m^t(G^{(a,a+1]}), \text{ and} \\
m^t(F^{\leq a}) &\leq m^t(E^{\leq a}) + m^t(G^{\leq a+\epsilon}) + r_\epsilon(t) m^t(E^{(a-1,a]}).
\end{align*}
where $r_\epsilon(t) := \frac{e^{\epsilon t}}{\sin(\pi \epsilon)} \max(e^{t},1)$.
\end{prop}

\begin{proof}
Up to shifting applying the action of $\bf{C}$ on $\Stab(\cD)$, we may assume $a=0$. By \Cref{L:pre_triangle_inequality}, it suffices to bound $g_t(Z(\im(\delta : G_0 \to E_{-1})))$. We apply \Cref{L:bound_gt_image} for $b=1-\epsilon$ and $a$ very close to $0$ to deduce that\endnote{We have used that $\min(\Im(Z(E^{\leq b})),\Im(Z(G))) \leq \Im(Z(G)) \leq |Z(G)|$.}
\begin{align*}
e^{-t}g_t(Z(\im(\delta))) &\leq \max \left(e^{-t} m^t(G_0),e^{-t} m^t(E_{-1}^{>1-\epsilon}) + \frac{e^{- \epsilon t}}{\sin(\pi \epsilon)} |Z(G_0)| \right)\\
&\leq m^t(E^{(-\epsilon,0]}) + \max \left(e^{-t} m^t(G_0),\frac{e^{- \epsilon t}}{\sin(\pi \epsilon)} |Z(G_0)| \right).
\end{align*}
We then apply \Cref{L:gt_inequality} to deduce that $|Z(G_0)| \leq m^t(G_0) e^{-t \phi(Z(G_0))} \leq m^t(G_0) \max(e^{-t},1)$. This gives the first inequality of the proposition.

For the second inequality, we apply \Cref{L:bound_gt_image} with $a=\epsilon$ and $b$ very close to $0$ to get
\begin{align*}
g_t(Z(\im(\delta))) &\leq \max\left(m^t(E_{-1}), m^t(G_0^{\leq \epsilon}) + \frac{e^{\epsilon t}}{\sin(\pi \epsilon)} |Z(E_{-1})|\right)\\
&\leq m^t(G^{(0,\epsilon]}) + \max \left(m^t(E_{-1}),\frac{e^{\epsilon t}}{\sin(\pi \epsilon)} |Z(E_{-1})| \right) \\
&\leq m^t(G^{(0,\epsilon]}) + \frac{e^{\epsilon t}}{\sin(\pi \epsilon)} \max\left(e^{-t},1\right) m^t(E_{-1}). \qedhere
\end{align*}
\end{proof}

\begin{proof}[Proof of \Cref{T:filtration_inequality}]
For each $i$, we consider the exact triangle $F_i \to E_i \to E_{i-1}$, with $E_n=E$ and $E_0=0$. If we let $F'_i := \fib(E_i \to E_{i-1}^{\leq a_{i-1}})$, then some manipulation of the octahedral axiom implies that\endnote{The octahedral axiom gives an exact triangle
\[
F_i \to F_i' \to E_{i-1}^{>a_{i-1}}.
\]
Applying the octahedral axiom to the composition $(F_i')^{>a_{i-1}} \to F_i' \to E_i$ gives exact triangles
\[
(F_i')^{>a_{i-1}} \to E_i \to E_{i-1}', \text{ and}
\]
\[
(F_i')^{\leq a_{i-1}} \to E_{i-1}' \to E_{i-1}^{\leq a_{i-1}}.
\]
The second triangle implies that $E_{i-1}' \in \cP(\leq a_{i-1})$, and thus the first triangle implies $(F_i')^{>a_{i-1}} \cong E_i^{>a_{i-1}}$ and $E_{i-1}' \cong E_i^{\leq a_{i-1}}$. Note also that $(F_i')^{\leq a_{i-1}} = \fib(E_i^{\leq a_{i-1}} \to E_{i-1}^{\leq a_{i-1}})$.}
\[
(F_i')^{>a_{i-1}} \cong E_i^{>a_{i-1}} \text{ and } (F_i')^{\leq a_{i-1}} = \fib(E_i^{\leq a_{i-1}} \to E_{i-1}^{\leq a_{i-1}}).
\]

\Cref{P:truncated_triangle_inequality} implies that
\[m^t(E_{i+1}^{>a_{i+1}}) = m^t((F_{i+1}')^{>a_{i+1}}) \leq m^t(F_{i+1}^{>a_{i+1}-\epsilon}) + (1+r_\epsilon(-t)) m^t(E_i^{>a_i}).
\]
If we let $C_0=0$ and $C_{i+1} := m^t(F_{i+1}^{>a_{i+1}-\epsilon}) +(1+r_\epsilon(-t)) C_{i}$, this shows that $m^t(E_{i}^{>a_{i}}) \leq C_{i}$ for all $i$. Next, \Cref{P:truncated_triangle_inequality} applied to the triangle $F_{i+1} \to F_{i+1}' \to E_{i}^{>a_{i}}$ implies that
\begin{align*}
m^t((F_{i+1}')^{\leq a_{i}}) &\leq m^t(E_{i}^{(a_{i},a_{i}+\epsilon]})+ (1+r_\epsilon(t)) m^t(F_{i+1}^{\leq a_{i}}) \\
&\leq m^t(E_i^{>a_i})+(1+r_\epsilon(t)) m^t(F_{i+1}^{\leq a_{i}}).   
\end{align*}
Using this we compute\endnote{The full chain of deductions is:
\begin{align*}
m^t(E_{i+1}) &= m^t(E_{i+1}^{>a_{i}}) + m^t(E_{i+1}^{\leq a_{i}}) \\
&= m^t((F_{i+1}')^{>a_{i}})+m^t(E_{i}') \\
&\geq m^t(F_{i+1}') + m^t(E_{i}) - C_{i} - (1+e^t) m^t((F_{i+1}')^{\leq a_{i}}) \\
&\geq m^t(F_{i+1}) - e^{-t} C_{i} + m^t(E_{i}) - C_{i} - (1+e^t) m^t((F_{i+1}')^{\leq a_{i}}) \\
&= m^t(F_{i+1}) + m^t(E_{i}) - (1+e^t)^2 e^{-t} C_{i} - (1+e^t) r_\epsilon m^t(F_{i+1}^{\leq a_{i}})
\end{align*}}
\begin{align*}
m^t(E_{i+1}) & \geq  m^t(F_{i+1}) + m^t(E_{i})
            - (1+e^t)^2 e^{-t} m^t(E_i^{>a_i}) +(1+e^t)(1+ r_\epsilon(t)) m^t(F_{i+1}^{\leq a_{i}}).
\end{align*}
This sets up an inductive argument bounding $m^t(E_i)$ below, starting from the fact that $m^t(E_1)=m^t(F_1)$. This inductive computation gives
\[
m^t(E_i) \geq \sum_{j=1}^i m^t(F_j) - (1+e^t)^2 e^{-t} M_i - (1+e^t) (1+r_\epsilon(t)) N_i,
\]
where
\[
M_i := \sum_{j=1}^{i-1} \frac{(1+r_\epsilon(-t))^{i-j}-1}{r_\epsilon(-t)} m^t(F_{j}^{>a_{j}-\epsilon}) \quad \text{ and } \quad N_i := \sum_{j=2}^i m^t(F_j^{\leq a_{j-1}}).
\]
So observing that the coefficients of the sum defining $M_n$ are all $\geq 1$, one can let
\[
C_{n,\epsilon,t} := (1+e^t) \max\left( (1+e^{-t}) \frac{(1+r_\epsilon(-t))^{n-1}-1}{r_\epsilon(-t)}, r_\epsilon(t) \right),
\]
and the right inequality of the lemma follows. The bound $m^t(E) \leq \sum_{i=1}^n m^t(F_i)$ follows from an iterated application of  \Cref{L:triangle_inequality}.
\end{proof}

In order to apply \Cref{T:filtration_inequality}, we will need a criterion for showing that the mass of $F_i$ outside of a given interval is small. For this, for any $t\in \bf{R}$ we let
\begin{equation} \label{E:cosh_expression} 
    c_{\sigma}^t(E) := \frac{1}{m(E)} \int \cosh t(\theta-\phi(E))\: \d m_{\sigma,E},  
\end{equation}
where $\phi(E)$ is the average phase of $E$ with respect to $\sigma$. The quantity $c^t_\sigma(E)$ reflects how concen\-trated the probability measure $\d m_E/m(E)$ is about its mean. Since $\cosh(\theta) \geq 1$ for all $\theta \in \bf{R}$, we have that $c^t_\sigma(E) \ge 1$ for any $E \in \mathfrak{d}$ with equality if and only if $E$ is $\sigma$-semistable.

\begin{lem}\label{L:cosh_bound_mass}
For a fixed pre-stability condition on $\cD$, $t>\lvert s\rvert>0$, and $\delta>0$, if $E \in \cD$ satisfies 
\begin{equation} \label{E:cosh_bound}
    \frac{c^{t}(E)}{m(E)} < \left( \frac{1}{2} + \sqrt{\frac{1}{4} + \epsilon \frac{\cosh(t\delta)}{\cosh(s\delta)}  - \epsilon} \right),
\end{equation}
for some $\epsilon>0$, then
\[
    m^s(E^{\leq \phi(E) - \delta}) + m^s(E^{\geq \phi(E)+\delta}) < \epsilon m(E) e^{s \phi(E)} \leq \epsilon m^s(E).
\]
\end{lem}
\begin{proof}
We will assume $s>0$, and show that the desired bound holds for both $s$ and $-s$. The function $\cosh(t \theta) / \cosh(s \theta)$ is even, $\geq 1$, and monotone increasing for $\theta > 0$. Therefore, the terms of $c_t(E)$ with $|\phi(E)-\theta| \geq \delta$ are bounded below by $\frac{\cosh(t \delta)}{\cosh(s\delta)} \cosh(s|\theta-\phi(E)|) m(\gr^\theta(E))$, and the other terms are bounded below by $\cosh(s|\theta-\phi(E)|) m(\gr^\theta(E))$. Let $E' := E^{\leq \phi(E) - \delta} \oplus E^{\geq \phi(E)+\delta}$, let $a := \frac{m_s(E')}{e^{s\phi(E)}} + \frac{m_{-s}(E')}{e^{-s\phi(E)}}$, and let $K$ denote the right-hand-side of \eqref{E:cosh_bound}. Then because $c_t(E)/c_s(E) \leq c_t(E)/m(E)$, we have
\[
c_s(E) - a + \frac{\cosh(t \delta)}{\cosh(s\delta)} a < c_t(E) < K c_s(E),
\]
which combined with the fact that $c_s(E) \leq c_t(E) < K m(E)$
\[
a < \frac{K(K-1)}{\frac{\cosh(t \delta)}{\cosh(s\delta)} - 1} m(E).
\]
Now $K$ was chosen precisely so that the right hand side of this inequality is $= \epsilon m(E)$. It follows that $a/m(E)<\epsilon$. But $a$ is a sum of two nonnegative numbers, so both terms $m^s(E')/ (m(E) e^{s\phi(E)})$ and $m^{-s}(E')/ (m(E) e^{-s\phi(E)})$  must satisfy this upper bound. The second inequality follows from Jensen's inequality \eqref{E:jensen}.
\end{proof}

\appendix

\section{Alternative proof of Lemma \ref{L:gt_inequality}}

Here is an alternative way to structure the proof of \Cref{L:gt_inequality}. We ultimately chose the proof in the main text because it involved less ``visual reasoning'', but we figured the following may also be of interest.

Let us first prove the claim when $\phi(w_i) > 0$ for all $i$ by induction on $m$. Note that condition (2) is automatic in this case, because $\Sigma w_m = \Sigma z_n$. Note that we may choose arbitrarily small perturbations of $z_\bullet$ and $w_\bullet$ that still satisfy the hypotheses of the lemma and such that $z_i \notin \bR w_j$ for any $i,j$. Because $g_t(z)$ is continuous on $\bC \setminus \bR_{\leq 0}$, for any fixed $m$ it suffices to prove the claim under this additional hypothesis. Also, by combining adjacent $w_i$ if necessary, we can assume that $\phi(w_i) > \phi(w_{i+1})$ for all $i$.

We will make use of \cite{ikeda_2021}*{Lem. 3.6}, which states that $g_t(u_1 + u_2) \leq g_t(u_1) + g_t(u_2)$ for any $u_1,u_2 \in \bH \cup \bR_{<0}$. However, if $u_1,u_2 \in \bC^*$ are any two points such that $\phi(u_2) \leq \phi(u_1) < \phi(u_2) + 1$, then we can multiply $u_1$ and $u_2$ by $e^{i\theta}$ for some $\theta \geq 0$ so that both lie in $\bH \cup \bR_{<0}$, and $g_t(u_1)$, $g_t(u_2)$, and $g_t(u_1+u_2)$ are all scaled by $e^{t\theta}$. We therefore have $g_t(u_1+u_2) \leq g_t(u_1) + g_t(u_2)$ in this more general setting, and we refer to this as the triangle inequality for $g_t$.

Let $\gamma(z_\bullet)$ be the piecewise-linear curve connecting $\Sigma z_0,\Sigma z_1, \Sigma z_2,\ldots$. The hypotheses of the lemma, along with the additional hypothesis that $z_i \notin \bR w_j$ for all $i,j$, implies that one of two possibilities occurs:
\begin{enumerate}[label=(\roman*)]
    \item the ray $\bR_{> 0} w_1$ intersects $\gamma(z_\bullet)$ at a unique interior point; or
    \item $m=1$, and $\bR_{> 0} w_1$ intersects $\gamma(z_\bullet)$ at the point $w_1 = \Sigma z_n$ and at most one other interior point.
\end{enumerate}
After subdividing if necessary, we can assume that the interior intersection point is $\Sigma z_r$ for some $r$. We will address each case separately:

\medskip
\noindent\textit{Case (ii):} This is the base case of our induction. Let $u_1=w_1$, and let $u_i = w_1-\Sigma z_{i-1}$ for $i=2,\ldots,n$,and $u_0 = w_1$. Then we can rewrite
\[
g_t(z_\bullet) - g_t(w_1) = \sum_{i=1}^{n-1} [g_t(z_i) + g_t(u_{i+1}) - g_t(u_i)].
\]
We observe in \Cref{fig:gt_inequality_base_case} that $g_t(u_i) \leq g_t(u_{i+1}) + g_t(z_i)$. Therefore $g_t(z_\bullet) - g_t(w_\bullet) \geq 0$ because each of the bracketed terms above is nonnegative.

\begin{figure}[h]
    \centering
\begin{tikzpicture}
\coordinate (v1) at (-1,0.5);
\coordinate (v2) at (1,-1.5);
\coordinate (v3) at (-2.5,-1);
\coordinate (v4) at (-4,0.5);
\coordinate (v5) at (-4,2.5);
\coordinate (v6) at (-2,3.5);
\coordinate (v7) at (0,3);
\coordinate (v8) at (1,2);
\coordinate (v9) at (0.5,1);

\draw[thick, orange] (v2) edge[->] node[auto] {$w_1$} (v1);
\draw[thick, gray] (v3) edge[->] node[auto] {$u_2$} (v1)
				(v4) edge[->] node[auto] {$u_3$} (v1)
				(v5) edge[->] node[auto] {$u_4$} (v1)
				(v6) edge[->] node[auto] {$u_5$} (v1)
				(v7) edge[->] node[auto] {$u_6$} (v1)
				(v8) edge[->] node[auto] {$u_7$} (v1);
\draw[thick, blue] (v2) edge[->] node[auto] {$z_1$} (v3)
				(v3) edge[->] node[auto] {$z_2$} (v4)
				(v4) edge[->] node[auto] {$z_3$} (v5)
				(v5) edge[->] node[auto] {$z_4$} (v6);
\draw[thick, purple] (v6) edge[->] node[auto] {$z_5$} (v7)
				(v7) edge[->] node[auto] {$z_6$} (v8)
				(v8) edge[->] node[auto] {$z_7$} (v9)
				(v9) edge[->] node[auto] {$z_8=u_8$} (v1);
\end{tikzpicture}
    \caption{Illustration of the base case for the induction in the proof of \Cref{L:gt_inequality}, when $m=1$. The fact that $\phi(u_{i+1})\leq \phi(z_i)< \phi(u_{i+1})+1$ and $u_i = u_{i+1} + z_i$ implies the triangle inequality $g_t(u_i) \leq g_t(u_{i+1})+g_t(z_i)$.}
    \label{fig:gt_inequality_base_case}
\end{figure}
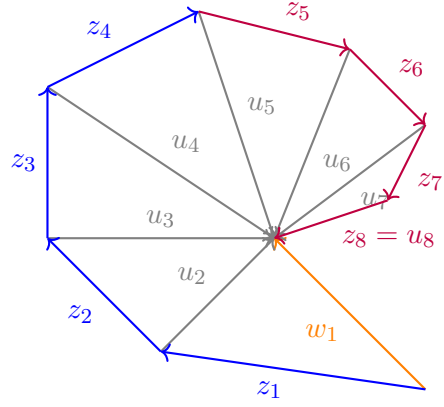

\medskip
\noindent\textit{Case (i):} Observe that $g_t(\Sigma z_r-w_1)+g_t(w_1) = g_t(\Sigma z_r)$, because all three complex numbers lie on the same ray. Using this we have
\begin{equation}\label{E:decompose_gt}
\begin{array}{rl}
g_t(z_\bullet) - g_t(w_\bullet) = &\sum_{i=2}^r \left[ g_t(\Sigma z_{i-1},z_i) - g_t(\Sigma z_i) \right] \\
&+ \left[ g_t(\Sigma z_r - w_1, z_{r+1},\ldots,z_n) - g_t(w_2,\ldots,w_m)\right].
\end{array}
\end{equation}
The hypotheses of the lemma imply that $\phi(z_{i+1}) \leq \phi(\Sigma z_i) \leq \phi(z_{i+1})+1$ for $i=1,\ldots,r$, and that each square-bracketed expression also satisfies the hypotheses of the lemma. See \Cref{fig:triangle_inequality} for a visualization of this division.

\begin{figure}[h]
    \centering
    \begin{tikzpicture}
        \coordinate (v1) at (0,0);
        \coordinate (v6) at (0.5,2);
        \coordinate (v2) at (-1.5,4) {};
        \coordinate (v3) at (0,5.25);
        \coordinate (v4) at (2.5,4.75);
        \coordinate (v5) at (3,4);
        \coordinate (v8) at (1.25,5);
        \coordinate (v7) at (2.5,3) {};
        \coordinate (v10) at  (-2,1);
        \coordinate  (v11) at (1.25,2.75);
         \draw[thick, blue] (v1) edge[->] node[anchor=north] {$z_1$} (v10) 
         				 (v10) edge[->] node[anchor=east] {$z_2$} (v2)
         				 (v2) edge[->] node[anchor=east] {$z_3$} (v3);
         \draw[thick, purple] (v3) edge[->] node[anchor=south] {$z_4$} (v8)
         				 (v8) edge[->] node[anchor=south] {$z_5$} (v4)
         				 (v4) edge[->] node[anchor=west] {$z_6$} (v5)
         				 (v5) edge[->] node[anchor=west] {$z_7$} (v7);
         \draw[thick, orange] (v1) edge[->] node[anchor=west] {$w_1$} (v6)
					(v6) edge[->] node[anchor=north west] {$w_2$} (v11)         
         				(v11) edge[->] node[anchor=north] {$w_3$} (v7);
         \fill[black] (v1) circle (1pt);
        \draw[thick, dashed, gray] (v6) edge[->] node[anchor=east] {$\Sigma z_4 -w_1$} (v8);
        \node[anchor=north] at (v1) {$0$};
        \node[anchor=north west] at (v8) {$\Sigma z_4$};
    \end{tikzpicture}

    \begin{tikzpicture}
        \coordinate (v6) at (1.5,2);
        \coordinate (v4) at (3.5,4.75);
        \coordinate (v5) at (4,4);
        \coordinate (v7) at (3.5,3) {};
        \coordinate (v8) at (2.25,5);
        \coordinate (v11) at (2.25,2.75);
        \draw[thick, blue] (v6) edge[->] node[fill=white] {$\Sigma z_4 - w_1$} (v8);
        \draw[thick, purple] (v8) edge[->] node[auto] {$z_5$} (v4)
         		(v4) edge[->] node[auto] {$z_6$} (v5)
         		(v5) edge[->] node[auto] {$z_7$} (v7);
        \draw[thick, orange] (v6) edge[->] node[anchor=north west] {$w_2$} (v11)
        		(v11) edge[->] node[anchor=north] {$w_3$} (v7);
        
        \coordinate (av1) at (-2,0);
        \coordinate (av10) at (-4,1);
        \coordinate (av2) at (-3.5,4);
        \draw[thick, blue] (av1) edge[->] node[auto] {$z_1$} (av10) 
         		(av10) edge[->] node[auto] {$z_2$} (av2);
        \draw[thick, orange] (av1) edge[->] node[fill=white] {$\Sigma z_2$} (av2);
        
        \coordinate (bv1) at (-1,0);
        \coordinate (bv2) at (-2.5,4);
        \coordinate (bv3) at (-1,5.25);
        \draw[thick, blue] (bv1) edge[->] node[fill=white] {$\Sigma z_2$} (bv2)
         		(bv2) edge[->] node[auto] {$z_3$} (bv3);
        \draw[thick, orange] (bv1) edge[->] node[fill=white] {$\Sigma z_3$} (bv3);
        
        \coordinate (cv1) at (0,0);
        \coordinate (cv3) at (0,5.25);
        \coordinate (cv8) at (1.25,5);
        \draw[thick, blue] (cv1) edge[->] node[near end,fill=white] {$\Sigma z_3$} (cv3);
        \draw[thick, purple] (cv3) edge[->] node[auto] {$z_4$} (cv8);
        \draw[thick, orange] (cv1) edge[->] node[fill=white] {$\Sigma z_4$} (cv8);
    \end{tikzpicture}

    \caption{Illustration of the inductive argument used in the proof of \Cref{L:triangle_inequality}. The curve $(z_1,\ldots,z_7)$ begins with a section on which $\Im(z)$ is monotone increasing (blue), followed by a section on which $\Im(z)$ is monotone decreasing (purple), whereas $\Im(z)$ is monotone increasing on the curve $(w_\bullet)$ (orange). This will always be the case under the hypotheses of \Cref{L:gt_inequality}. Note that a vertex $\Sigma z_4$ has been added where the ray $\bR_{\geq 0} w_1$ hits the top curve. The calculation of $g_t(z_\bullet)-g_t(w_\bullet)$ for the top diagram adds the value of $g_t$ for the blue and purple arrows and subtracts the value for the orange arrows. The bottom diagram shows how this sum is decomposed as a sum of smaller expressions of the same form.}
    \label{fig:triangle_inequality}
\end{figure}
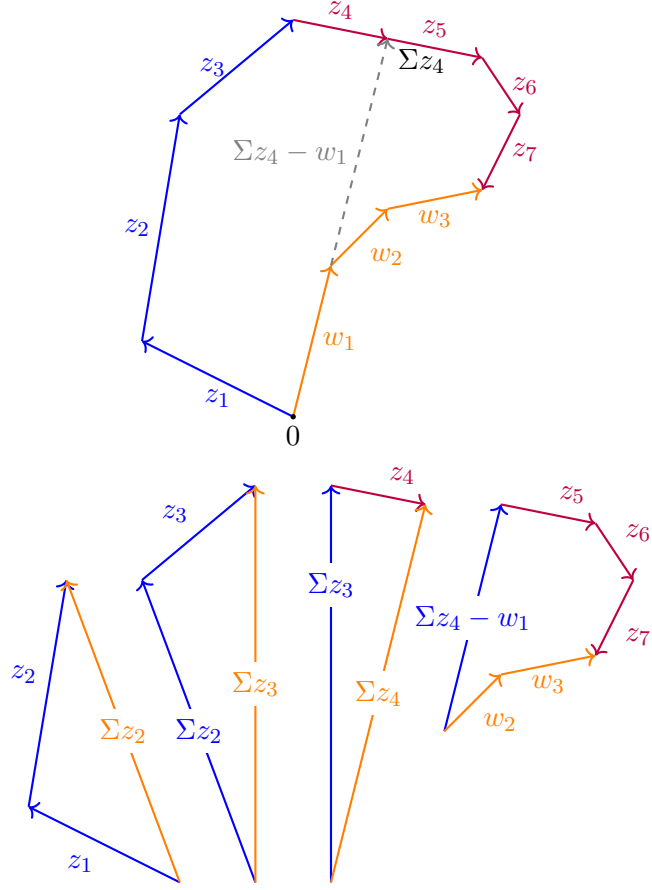

The inductive hypothesis implies that $g_t(\Sigma z_r - w_1, z_{r+1},\ldots,z_n) - g_t(w_2,\ldots,w_m) \geq 0$. Also, the bounds $\phi(\Sigma z_{i-1}) \leq \phi(z_i) < \phi(\Sigma z_{i-1}) + 1$ imply that $g_t(\Sigma z_i) := g_t(z_i + \Sigma z_{i-1}) \leq g_t(z_i)+g_t(\Sigma z_{i-1})$ for $i=2,\ldots,r$. We conclude that every bracketed term on the right-hand-side of \eqref{E:decompose_gt} is nonnegative, and $g_t(z_1,\ldots,z_n) - g_t(w_1,\ldots,w_m) \geq 0$ as desired.

In the setting where $\phi(w_i) \leq 0$ for all $i$, the proof of the lemma is completely analogous. For the base case of the induction, when $m=1$, we use the decomposition
\[
g_t(z_\bullet)-g_t(w_1) = \sum_{i=2}^n [g_t(\Sigma z_{i-1}) + g_t(z_i) - g_t(\Sigma z_i)].
\]
Under the hypotheses (2), one has $\phi(z_i) \leq \phi(\Sigma z_{i-1}) < \phi(z_i)+1$ for all $i$, so we conclude that every square bracketed expression is nonnegative.

For the inductive step, if $m>1$, then the ray $\Sigma z_n - \bR_{>0} w_m$ must intersect the curve $\gamma(z_\bullet)$ at a unique interior point, which we can assume is of the form $\Sigma z_r$ for some $r$. Then if we let $u_i = \Sigma z_n - \Sigma z_i$, we can write
\begin{align*}
g_t(z_\bullet)-g_t(w_\bullet) &= [g_t(z_1,\ldots,z_{r-1},u_r)-g_t(w_1,\ldots,w_{m-1})] 
+ \sum_{i=r}^{n-2} [g_t(u_{i+1},z_{i+1}) - g_t(u_i)]
\end{align*}
Once again, the geometry of the curve $\gamma(z_\bullet)$ guarantees that $\phi(u_{i+1}) \leq \phi(z_{i+1}) < \phi(u_{i+1})+1$ for $i=r,\ldots,n-2$. So, the inductive hypothesis implies the first bracketed terms are nonnegative, and the triangle inequality for $g_t$ implies that the remaining terms on nonnegative.

We can now prove the general statement: The ray $\Sigma w_s + \bR_{>0} \cdot w_s$ intersects the curve $\gamma(z_\bullet)$ in a single point, and by subdividing we can assume that it is $\Sigma z_p$ for some $p$. Then we write
\begin{align*}
g_t(z_\bullet) - g_t(w_\bullet) &= [g_t(z_1,\ldots,z_p)-g_t(w_1,\ldots,w_s,\Sigma z_p - \Sigma w_s)] \\
& + [g_t(\Sigma z_p - \Sigma w_s,z_{p+1},\ldots,z_n) - g_t(w_{s+1},\ldots,w_m)]
\end{align*}
Both bracketed terms above satisfy the hypotheses of the lemma. The first term has $\phi(w_i)>0$ for all $i$, and the second term has $\phi(w_i) \leq 0$ for all $i$, so we have shown that both bracketed terms are nonnegative.

\printendnotes

\bibliography{refs}{}

@article {Br07,
    AUTHOR = {Bridgeland, Tom},
     TITLE = {Stability conditions on triangulated categories},
   JOURNAL = {Ann. of Math. (2)},
  FJOURNAL = {Annals of Mathematics. Second Series},
    VOLUME = {166},
      YEAR = {2007},
    NUMBER = {2},
     PAGES = {317--345},
      ISSN = {0003-486X},
   MRCLASS = {14F05 (18E30)},
  MRNUMBER = {2373143},
MRREVIEWER = {Leovigildo M. Alonso Tarrio},
       DOI = {10.4007/annals.2007.166.317},
       URL = {https://doi.org/10.4007/annals.2007.166.317},
}

@misc{stacks-project,
    shorthand    = {Stacks},
    author       = {The {Stacks Project Authors}},
    title        = {\textit{Stacks Project}},
    howpublished = {\url{https://stacks.math.columbia.edu}},
    year         = {2018},
  }

@article{GKRpaper04,
doi = {10.1070/IM2004v068n04ABEH000497},
url = {https://dx.doi.org/10.1070/IM2004v068n04ABEH000497},
year = {2004},
month = {aug},
publisher = {},
volume = {68},
number = {4},
pages = {749},
author = {A L Gorodentsev and  S A Kuleshov and  A N Rudakov},
title = {t-stabilities and t-structures on triangulated categories},
journal = {Izvestiya: Mathematics}
}

@misc{KS08,
  
  url = {https://arxiv.org/abs/0811.2435},
  
  author = {Kontsevich, Maxim and Soibelman, Yan},
  
  title = {Stability structures, motivic {D}onaldson-{T}homas invariants and cluster transformations},
  
  publisher = {arXiv},
  
  year = {2008},
  
  copyright = {arXiv.org perpetual, non-exclusive license}
}

@article{BridgelandK3,
 author = {Bridgeland, Tom},
 title = {Stability conditions on {K}3 surfaces},
 fjournal = {Duke Mathematical Journal},
 journal = {Duke Math. J.},
 issn = {0012-7094},
 volume = {141},
 number = {2},
 pages = {241--291},
 year = {2008},
 language = {English},
 doi = {10.1215/S0012-7094-08-14122-5},
 zbMATH = {5237682},
 Zbl = {1138.14022}
}

@article{Bayer_short,
	Author = {Bayer, Arend},
	Da = {2019/12/01},
	Doi = {10.1007/s00208-019-01900-w},
	Id = {Bayer2019},
	Isbn = {1432-1807},
	Journal = {Mathematische Annalen},
	Number = {3},
	Pages = {1597--1613},
	Title = {A short proof of the deformation property of {B}ridgeland stability conditions},
	Ty = {JOUR},
	Url = {https://doi.org/10.1007/s00208-019-01900-w},
	Volume = {375},
	Year = {2019}
}

@article{ikeda_2021, 
title={Mass growth of objects and categorical entropy}, 
volume={244}, 
DOI={10.1017/nmj.2020.9}, 
journal={Nagoya Mathematical Journal}, 
publisher={Cambridge University Press}, 
author={Ikeda, Akishi}, year={2021}, 
pages={136–157}
}

@article{quasiconvergence,
      title={Quasi-convergence of stability conditions}, 
      author={Daniel Halpern-Leistner and Jeffrey Jiang and Antonios-Alexandros Robotis},
      year={2026},
    journal = {Selecta Mathematica, to appear},
}

@incollection{DynamicsDHKK,
 author = {Dimitrov, George and Haiden, Fabian and Katzarkov, Ludmil and Kontsevich, Maxim},
 title = {Dynamical systems and categories},
 booktitle = {The influence of {S}olomon {L}efschetz in geometry and topology},
 isbn = {978-0-8218-9494-1},
 pages = {133--170},
 year = {2014},
 publisher = {Providence, RI: American Mathematical Society (AMS)},
 language = {English},
 zbMATH = {6520455},
 Zbl = {1348.18017}
}

@misc{laxstability,
      title={Partial compactification of stability manifolds by massless semistable objects}, 
      author={Nathan Broomhead and David Pauksztello and David Ploog and Jon Woolf},
      year={2022},
      eprint={2208.03173},
      archivePrefix={arXiv},
      primaryClass={math.RT},
      url={https://arxiv.org/abs/2208.03173}, 
}

@misc{ThurstonBDL,
      title={A {T}hurston compactification of the space of stability conditions}, 
      author={Asilata Bapat and Anand Deopurkar and Anthony M. Licata},
      year={2023},
      eprint={2011.07908},
      archivePrefix={arXiv},
      primaryClass={math.RT},
      url={https://arxiv.org/abs/2011.07908}, 
}

@article{ThurstonKKO,
   title={Thurston compactifications of spaces of stability conditions on curves},
   ISSN={1088-6850},
   url={http://dx.doi.org/10.1090/tran/9104},
   DOI={10.1090/tran/9104},
   journal={Transactions of the American Mathematical Society},
   publisher={American Mathematical Society (AMS)},
   author={Kikuta, Kohei and Koseki, Naoki and Ouchi, Genki},
   year={2024},
   month=feb 
}

@misc{DeopurkarThurstonk3,
      title={The Thurston compactification of the stability manifold of a generic analytic {K}3 surface}, 
      author={Anand Deopurkar},
      year={2025},
      eprint={2505.14991},
      archivePrefix={arXiv},
      primaryClass={math.AG},
      url={https://arxiv.org/abs/2505.14991}, 
}

@misc{Bolognesecompactification,
      title={A local compactification of the {B}ridgeland stability manifold}, 
      author={Barbara Bolognese},
      year={2020},
      eprint={2006.04189},
      archivePrefix={arXiv},
      primaryClass={math.AG},
      url={https://arxiv.org/abs/2006.04189}, 
}

@misc{augmented,
      title={The space of augmented stability conditions}, 
      author={Daniel Halpern-Leistner and Antonios-Alexandros Robotis},
      year={2025},
      eprint={2501.00710},
      archivePrefix={arXiv},
      primaryClass={math.AG},
      url={https://arxiv.org/abs/2501.00710}, 
}

@inproceedings{SouriauGroupes,
 author = {Souriau, J.-M.},
 title = {Groupes diff{\'e}rentiels de physique math{\'e}matique},
 year = {1984},
 language = {English},
 howpublished = {Feuilletages et quantification g{\'e}om{\'e}trique, {Journ}. lyonnaises {Soc}. math. {France} 1983, {S{\'e}min}. sud-rhodanien {G{\'e}om}. {II}, 73-119 (1984).},
 zbMATH = {3860012},
 Zbl = {0541.58002}
}

@article{BPPDiscrete,
 author = {Broomhead, Nathan and Pauksztello, David and Ploog, David},
 title = {Discrete triangulated categories},
 fjournal = {Bulletin of the London Mathematical Society},
 journal = {Bull. Lond. Math. Soc.},
 issn = {0024-6093},
 volume = {50},
 number = {1},
 pages = {174--188},
 year = {2018},
 language = {English},
 doi = {10.1112/blms.12125},
 zbMATH = {6846530},
 Zbl = {1391.18016}
}
\bibliographystyle{plain}

\end{document}